\documentclass{amsart}

\usepackage{amsmath}
\usepackage{amssymb}
\usepackage{graphicx}
\usepackage[colorinlistoftodos]{todonotes}
\usepackage[colorlinks=true, allcolors=blue]{hyperref}
\usepackage{tikz}
\usepackage{endnotes}
\usepackage{mathtools}
\usepackage{multirow}
\usepackage{float}
\usepackage{lipsum}

\theoremstyle{definition}

\newcommand{\R}{\mathbb{R}}

\newcommand{\N}{\mathbb{N}}

\newcommand{\E}{\mathcal{E}}
\newcommand{\Si}{\mathcal{S}}
\newcommand{\Siu}{{\Si_0}(u)}
\newcommand{\A}{\mathcal{A}(u)}
\newcommand{\la}{\langle}
\newcommand{\ra}{\rangle}

\newcommand{\Wbar}{\overline{W}}
\newcommand{\mubar}{\overline{\mu}}
\newtheorem{theorem}{Theorem}[section]

\newtheorem{lemma}[theorem]{Lemma}
\newtheorem{proposition}[theorem]{Proposition}
\newtheorem{assumption}[theorem]{Assumption}
\newtheorem{definition}{Definition}[section]

\newtheorem{convention}{Convention}[section]

\DeclarePairedDelimiter\floor{\lfloor}{\rfloor}
\numberwithin{equation}{section}

\title{Rectifiability of the Singular Set of Harmonic Maps into Buildings}
\author{Ben K. Dees}

\begin{document}
\begin{titlepage}
\begin{center}

\maketitle

Johns Hopkins University\\
Department of Mathematics\\
bdees1@jh.edu

\begin{abstract}
In Gromov and Schoen's \cite{GS}, a notion of energy-minimizing maps into singular spaces is developed.  Using some analogues of classical tools, such as Almgren's frequency function, they focus on harmonic maps into $F$-connected complexes, and demonstrate that the singular sets of such functions are closed and of codimension $2$.  More recent work on related problems, such as Naber and Valtorta's \cite{NV}, uses the study of monotone quantities such as the frequency function to prove stronger rectifiability results about the singular sets in these contexts.  In this paper, we adapt this more recent work to demonstrate that an energy-minimizing map from a Euclidean space $\R^m$ to an $F$-connected complex has a singular set which is $(m-2)$-rectifiable.  We also obtain Minkowski bounds on the singular sets of these maps, and in particular use these bounds to demonstrate the local finiteness of the $(m-2)$-dimensional Hausdorff measure of these singular sets.
\end{abstract}

\vspace{10mm}


\vspace{10mm}

{\bf Declarations}\\
Funding: Not applicable.\\
Conflicts of Interest: Not applicable.\\
Availability of data: Not applicable.\\
Code availability: Not applicable.\\

\vspace{10mm}

{\bf Acknowledgements.}  The author would like to thank Zahra Sinaei for sharing her insights and perspective on this subject.  He would also like to thank Christine Breiner for her useful suggestions, which have improved the exposition of this work.
\end{center}
\end{titlepage}

\section{Introduction}\label{Intro0}

In \cite{GS}, Gromov and Schoen develop a theory of harmonic maps into nonpositively curved spaces, showing that some of the regularity of ordinary harmonic maps into nonpositively curved manifolds extends to this more general case (e.g. local Lipschitz continuity and the monotonicity of Almgren's frequency function).   For certain spaces, namely $F$-connected complexes, they moreover show stronger regularity properties.  In particular, they show that the singular sets of harmonic maps into $F$-connected complexes have codimension $2$; everywhere else the maps are given by analytic maps into Euclidean spaces (these Euclidean spaces being the ``flats" or ``apartments" of the complexes).  This codimension bound and control on the singular set allow them to use a Bochner formula, as in the smooth setting.

This ability to transfer classical results into a singular setting allows Gromov and Schoen to show a $p$-adic superrigidity result for certain lattices.  Similar rigidity results for Teichm\"uller space and hyperbolic complexes, due to Daskalopoulos and Mese in \cite{DMRigidTeich} and \cite{DMVSuperrigid}, again use analogues of classical methods.  Thus, they rely on {\em a priori} control of the singular set, because the maps involved must be sufficiently regular on a sufficiently large portion of the domain for classical methods to be relevant.  The necessary {\em a priori} estimates, in \cite{DMEssReg} and \cite{DMDMComplex}, give bounds on the Hausdorff dimension of the singular sets.

However, \cite{GS} leaves open the question of whether their singular sets are rectifiable, as one might hope.  Azzam and Tolsa, in \cite{AT}, show an equivalent condition to rectifiability involving a quantity called the Jones $\beta_2$-numbers.  This is a key result which we use to show that the singular set of a harmonic map into an $F$-connected complex is $(m-2)$-rectifiable.  Drawing from the quantitative stratification results of Cheeger and Naber in \cite{CN}, Naber and Valtorta prove a rectifiable-Reifenberg result in \cite{NV} which provides important volume bounds in one part of this paper.  Similar results have been developed in e.g. \cite{ENV16, ENV19}.  The method of Naber and Valtorta provides a framework which de Lellis, Marchese, Spadaro, and Valtorta use to study $Q$-valued functions in \cite{DMSV}, and which Alper uses in \cite{A} to study an optimal partition problem.  We draw special attention to this work of Alper, \cite{A}, because it shows that the singular sets of harmonic maps into one-dimensional $F$-connected complexes (trees) are $(m-2)$-rectifiable.  In this paper, we adapt the framework of Naber-Valtorta to our special case, following the arguments of \cite{DMSV}.  Then, using the aforementioned results of \cite{AT} and \cite{NV}, we show that the singular set of a harmonic map into an $F$-connected complex is $(m-2)$-rectifiable.  

In particular, our main results are the following theorems, in which $\Si(u)$ denotes the singular set of $u$ and $\Siu$ denotes a particular subset of the singular set, the set of singular points which have high order in the sense of \cite{GS}.  (We shall define all of these terms more precisely in the following section.)

\begin{theorem}\label{Thm1.3}
The singular set $\Si(u)$ of a minimizing map $u:\Omega\to X$, for $\Omega\subset\R^m$ and $X$ an $F$-connected complex, is $(m-2)$-countably rectifiable.
\end{theorem}

In the following theorem, and throughout this paper, we use the notation $B_r(E)$ to denote the tubular neighborhood of radius $r$ about a set $E\subset\R^m$, that is,
\[
B_r(E):=\{y\in\R^m:d(x,y)<r\text{ for some }x\in E\}.
\]
Additionally, $\mathcal{H}^{m-2}$ denotes the $(m-2)$-dimensional Hausdorff measure on $\R^m$.

\begin{theorem}\label{Thm1.2}
Let $u:\Omega\to X$ be a minimizing map, where $\Omega\subset\R^m$ and $X$ is an $F$-connected complex.  Then for any compact $K\subset\Omega$, we have $\mathcal{H}^{m-2}(\Siu\cap K)<\infty$ and indeed we have the Minkowski-type estimate
\begin{equation}
|B_r(\Siu\cap K)|\leq C(K,u)r^2
\end{equation}
for $r<1$.  The set $\Siu$ is moreover $(m-2)$-countably rectifiable; that is, it is covered by a countable collection of $C^1$ surfaces of dimension $m-2$ and a set of $\mathcal{H}^{m-2}$ measure zero.
\end{theorem}

At the end of Section \ref{Prelim} we shall see how Theorem \ref{Thm1.2} implies Theorem \ref{Thm1.3}; the proof relies on a number of results from \cite{GS} that we shall summarize there.

\section{Preliminaries}\label{Prelim}

In this section we summarize the key definitions and theorems used in this paper.  This is intended to motivate our approaches and methods, and to provide the key context for the main theorems.  

In Subsection \ref{subsec:GS}, we focus on definitions and results of \cite{GS} and \cite{KSSobolev} which provide context for both general nonpositively curved spaces and $F$-connected complexes in particular.  The notion of a conical complex is discussed in Subsection \ref{subsec:cone}.  These complexes are slightly better behaved than general nonpositively curved complexes for the purposes of some variational results, and are hence a useful tool in the work of this paper.  We then introduce smoothed versions of the energy, height, and order functionals in the Subsection \ref{subsec:smoothedfun}.  Subsection \ref{subsec:key} focuses on a few key results and showing how these key results imply our main theorems.  Finally, we outline the paper, introducing the needed results of \cite{AT} and \cite{NV} and giving an overview of our arguments.

\subsection{Nonpositively Curved and $F$-Connected Complexes}\label{subsec:GS}

Here, we work with simply connected locally compact Riemannian simplicial complexes embedded in Euclidean space, as in \cite{GS}.  We consider complexes embedded into Euclidean space in such a way that the induced Riemannian metric on each simplex is the appropriate Riemannian metric.  Working with complexes embedded in this way is primarily for convenience.  We direct the reader to \cite{KSSobolev} for an exposition on nonpositively curved spaces which makes no use of such embeddings, and is hence more technical, but more general.

\begin{definition}
Throughout this subsection, we shall say that $X$ is an {\bf admissible complex} if $X$ is a locally compact Riemannian simplicial complex.
\end{definition}


The following preliminary discussion is based on Sections 1 and 2 of \cite{GS}, but we reproduce it here for the reader's convenience.

For an admissible complex $X$ embedded in $\R^N$, a smooth $m$-dimensional Riemannian manifold $(M,g)$, and $\Omega\subset M$ an open domain with smooth boundary we define the Sobolev space $H^1(\Omega,X)$ by

\[
H^1(\Omega,X)=\{u\in H^1(\Omega,\R^N): u(x)\in X\text{ a.e. }x\in\Omega\}.
\]

Here, $H^1(\Omega,\R^N)$ denotes the Sobolev space of of $L^2$ functions $u:\Omega\to\R^N$ with $L^2$ weak first derivatives.

\begin{definition}
For a function $u\in H^1(\Omega,\R^N)$ (and in particular for a function $u\in H^1(\Omega,X)$, our main context) the {\bf energy density} of $u$ is the $L^2$ function
\[
|\nabla u|^2:=\sum_{i,j=1}^mg^{ij}\Big\langle\frac{\partial u}{\partial x^{i}},\frac{\partial u}{\partial x^{j}}\Big\rangle
\]
where $\langle\cdot,\cdot\rangle$ denotes the dot product in $\R^N$.

In this paper we will usually work with the case where $M=\R^m$; in this case the energy density reduces to 
\[
|\nabla u|^2=\sum_{i=1}^m\Big|\frac{\partial u}{\partial x^i}\Big|^2.
\]
\end{definition}

\begin{definition}\label{def:energy}
The {\bf energy} of $u\in H^1(\Omega,X)$ on a set $A\subset\Omega$ is
\[
E_u(A):=\int_{A}|\nabla u|^2 d\mu_g
\] 
where $\mu_g$ denotes the Riemannian volume measure on $M$.  (Note that this is simply the usual Lebesgue measure if $M=\R^m$.)

If $A=B_r(x)$, the geodesic ball of radius $r$ about $x$, we will usually write $E_u(x,r):=E_u(B_r(x))$.  We also frequently omit the subscripts $u$ whenever the map is clear from context, or when only one map is under consideration.
\end{definition}

\begin{definition}\label{def:harmonic}
We say that $u$ is {\bf energy-minimizing on $\Omega$} (or {\bf minimizing} or {\bf harmonic}) if for any $v\in H^1(\Omega,X)$ so that $u=v$ on $\partial\Omega$, $E_u(\Omega)\leq E_v(\Omega)$.  (We say that two maps are equal on the boundary of $\Omega$ if they have the same trace, in the usual manner for Sobolev functions.)
\end{definition}

A preliminary result of \cite{GS} shows that for any $\phi\in H^1(\Omega,X)$, there is a unique energy minimizing $u$ so that $u=\phi$ on $\partial\Omega$.  Additionally, they prove that an energy minimizing map $\phi$ from $[0,1]$ minimizes length among Lipschitz curves between $\phi(0)$ and $\phi(1)$.  As in the fully smooth setting we have a term for such maps.

\begin{definition}
We call an energy-minimizing map $\phi:[0,1]\to X$ a {\bf geodesic}.
\end{definition}

Note that we require geodesics to be globally energy-minimizing here as opposed to the local condition often considered in the smooth setting.  Because of the uniqueness of energy-minimizing maps, any two geodesics between $P$ and $Q$ are the same up to reparametrization.  In \cite{GS} it is also shown that geodesics have constant speed so we may sensibly speak of ``unit speed geodesics" for example.

\begin{definition}
We define the {\bf distance} between points $P,Q\in X$ to be the length of the geodesic between them, and denote this as $d_X(P,Q)$.  It can be readily verified that (as in the smooth setting) this is a metric on $X$ which is compatible with the topology of $X$ as a Riemannian simplicial complex.
\end{definition}

In this context, \cite{GS} proves a monotonicity formula for Almgren's frequency function.

\begin{definition}\label{def:height}
For a minimizing map $u$, a ball $B_r(x)\subset\Omega$, and a point $Q\in X$, we define the {\bf height function} of $u$ on the ball (with respect to $Q$) as
\[
H_u(x,r,Q):=\int_{\partial B_r(x)}d^2(u(y),Q)d\Sigma(y)
\]
where $\Sigma$ denotes the surface measure of $\partial B_r(x)$.

Additionally, we define the {\bf frequency} of $u$ on $B_r(x)$ (with respect to $Q$) as
\[
I_u(x,r,Q):=\frac{rE_u(x,r)}{H_u(x,r,Q)}.
\]
\end{definition}

Gromov and Schoen show that this is approximately monotonic, in the precise sense that for some constant $c$ depending only on the metric $g$ of the domain manifold $(M,g)$, we have
\[
\frac{d}{dr}\Big[e^{cr^2}I_u(x,r,Q)\Big]\geq0.
\]
When $M=\R^m$ with the usual Euclidean metric, the constant $c$ is zero, so the frequency is increasing.  (We shall see a similar result for the smoothed frequency, later.)

From this result (equation (2.5) of \cite{GS}), they define a notion of ``order." 

\begin{definition}\label{def:order}
We define the function $\text{Ord}_u(x,r,Q):=e^{cr^2}I_u(x,r,Q)$.  In particular, if $\Omega$ is a Euclidean domain, we take $c=0$ so that $\text{Ord}_u=I_u$.  In \cite{GS} it is shown that there is a unique $Q_{r,x}$ maximizing this for fixed $x,r$, and that the limit
\[
\lim_{r\to0}\text{Ord}_u(x,r,Q_{x,r})
\]
exists.  The resulting function
\[
\text{Ord}_u(x):=\lim_{r\to0}\text{Ord}_u(x,r,Q_{x,r})
\]
is called the {\bf order function}.  This function is also shown to be upper semicontinuous and at least $1$ for all $x\in\Omega$.

We remark that the technical choice of $Q_{x,r}$ as a maximizer is only necessary because \cite{GS} has not---at this point---established that minimizing maps are continuous; they do this via the analysis of the order function.  {\em A posteriori}, if $u$ is continuous, we may instead define the order function as
\begin{equation}\label{def:ctsord}
\text{Ord}_u(x):=\lim_{r\to0}\text{Ord}_u(x,r,u(x)).
\end{equation}
This fact is shown just after the proof of Proposition 3.1 in \cite{GS}.  Because we will be working with continuous minimizers, we often use equation (\ref{def:ctsord}) as the definition of the order.
\end{definition}

To prove that minimizers are continuous, we must turn to the main context for the results of \cite{GS} and the next section of this paper.  Just as with classical harmonic maps, we can prove better regularity results when the target space has nonpositive sectional curvature.  However, we must define a notion of curvature that does not require the space to be smooth.

Suppose that $X$ is a simply connected admissible complex, and let $P_0,P_1,P_2\in X$.  Consider the geodesic $\gamma$ between $P_0$ and $P_1$, with $\gamma(0)=P_0$ and $\gamma(\ell)=P_1$, where $\ell=d(P_0,P_1)$.  This in particular means that $\gamma$ has unit speed.  Write
\[
D_X(s)=d^2_X(\gamma(s),P_2)
\]
to denote the squared distance from $\gamma(s)$ to $P_2$.  Geometrically, this is the squared distance from $P_2$, which is one vertex of the triangle $P_0,P_1,P_2$ to the opposite side of this triangle.  We denote by $D_E(s)$ the solution to the differential equation $D_E''(s)=2$ with $D_E(0)=D_X(0)$ and $D_E(\ell)=D_X(\ell)$.  Geometrically, if $A_0,A_1,A_2$ is a triangle in $\R^2$ with
\[
d_X(P_i,P_j) = |A_i-A_j|
\]
for all $i,j\in\{0,1,2\}$, $D_E$ is the squared distance from $A_2$ to the opposite side of the triangle formed by $A_0,A_1,A_2$.

\begin{definition}
We say that a simply connected admissible complex $X$ is a {\bf non-positively curved complex} or that $X$ is {\bf non-positively curved} if for any such selection of three points, and for any $s\in[0,\ell]$, $D_X(s)\leq D_E(s)$.  Geometrically, this says that when we compare triangles in the two spaces, the sides of triangles in a non-positively curved space ``bulge inwards" more than (or precisely, {\em at least as much} as) triangles in Euclidean space.
\end{definition}

A brief sketch to illustrate this definition is produced in Figure \ref{fig:NPC}.

\begin{figure}\centering
\begin{tikzpicture}
\filldraw[black] (0,0) circle (1pt) node [anchor=north]{$A_0$};
\filldraw[black] (2,1) circle (1pt) node [anchor=west]{$B_0$};
\filldraw[black] (1,2) circle (1pt) node [anchor=east]{$C_0$};

\filldraw[black] (-5,0) circle (1pt) node [anchor=north]{$A_-$};
\filldraw[black] (-3,1) circle (1pt) node [anchor=west]{$B_-$};
\filldraw[black] (-4,2) circle (1pt) node [anchor=east]{$C_-$};

\filldraw[black] (5,0) circle (1pt) node [anchor=north]{$A_+$};
\filldraw[black] (7,1) circle (1pt) node [anchor=west]{$B_+$};
\filldraw[black] (6,2) circle (1pt) node [anchor=east]{$C_+$};

\draw[black] (0,0) -- (2,1);
\draw[black] (2,1) -- (1,2);
\draw[black] (1,2) -- (0,0);

\draw[black] (-5,0) .. controls (-4,.875) .. (-3,1);
\draw[black] (-3,1) .. controls (-3.75, 1.3) .. (-4,2);
\draw[black] (-4,2) .. controls (-4.125,1) .. (-5,0);

\draw[black] (5,0) .. controls (6.25,.375) .. (7,1);
\draw[black] (7,1) .. controls (6.5,1.8) .. (6,2);
\draw[black] (6,2) .. controls (5.4,1.4) .. (5,0);
\end{tikzpicture}
\caption{Three triangles are shown.  The central triangle has straight sides, corresponding to flat space (zero curvature).  The one to the left has its sides bowed inwards, so that each vertex is closer to the opposite side than the flat case; this triangle corresponds to a negatively curved space.  The rightmost triangle has its sides bowed outwards, so that each vertex is further from the opposite side than in the flat case; this triangle corresponds to a positively curved space.  A nonpositively curved space cannot have geodesic triangles like the rightmost one.}
\label{fig:NPC}
\end{figure}
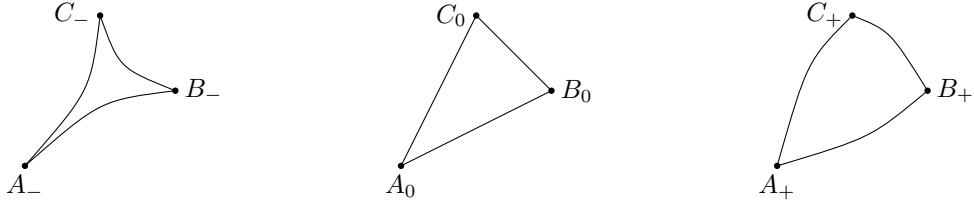

The following result, from \cite{KSSobolev}, makes critical use of the assumption of nonpositive curvature.

\begin{theorem}[Theorem 2.4.6 of \cite{KSSobolev}]\label{GS2.3}
Let $\Omega\subset(M,g)$ be a Lipschitz Riemannian domain, and let $X$ be a nonpositively curved space.  If $u:\Omega\to X$ is an energy minimizing map, then $u$ is a locally Lipschitz continuous function in the interior of $\Omega$.  The local Lipschitz constant at $x\in\Omega$ is bounded in terms of the dimension of $M$, the metric $g$, the total energy of $u$, and the distance from $x$ to $\partial\Omega$.
\end{theorem}



On occasion, it will be useful for us to recall the following result of \cite{GS} as well.

\begin{proposition}[Proposition 3.4 of \cite{GS}]\label{prop:GS3.4}
If $\Omega$ is connected, and $u:\Omega\to X$ is a minimizing map which is constant on an open subset of $\Omega$, then $u$ is identically constant on $\Omega$.
\end{proposition}

Additionally, in a nonpositively curved complex $X$, we have the following family of maps $X\to X$.

\begin{definition}
For any $P_0\in X$ and for any $\lambda\in[0,1]$, we shall define the {\bf retraction towards $P_0$ by $\lambda$}, $R_{P_0,\lambda}:X\to X$.

To define $R_{P_0,\lambda}(Q)$, let $\gamma:[0,1]\to X$ be the geodesic with $\gamma(0)=P_0$ and $\gamma(1)=Q$.  Then we let
\[
R_{P_0,\lambda}(Q)=\gamma(\lambda).
\]

This map is Lipschitz with Lipschitz constant at most $\lambda$; this can be verified by repeated applications of the nonpositive curvature condition.  Further, the family of maps $R_{P_0,\lambda}$ are a homotopy from the constant map at $P_0$ to the identity map.  (In particular, the nonpositively curved complexes considered here are all contractible.)
\end{definition}

Finally, we further specialize to a class of complexes where \cite{GS} shows further results.

\begin{definition}\label{def:Fcon}
We say that a nonpositively curved complex $X$ of dimension $k$ is {\bf $F$-connected} if any two adjacent simplices $S,S'$ are contained in a totally geodesic subcomplex $X_{S,S'}$ which is isometric to a subset of the Euclidean space $\R^k$.

A {\bf $k$-flat} of $X$ is a totally geodesic subcomplex of $X$ isometric to a subset of $\R^k$.

For a point $P\in X$ an $F$-connected complex, we denote the {\bf tangent cone at $P$} by $X_P$.  Observe that a neighborhood of $P$ in $X$ is isometric to a neighborhood of the origin in $X_P$.  For this reason we will sometimes (in a mild abuse of notation) blur the distinction between flats of $X_P$ and flats of $X$ that contain $P$.
\end{definition}

If a minimizing map has its image contained in some particular $k$-flat of $X$, we may regard it as a minimizing map into $\R^k$; it hence inherits all regularity properties of classical harmonic maps.  This motivates the following division of the domain into the regular points and the singular points.

\begin{definition}\label{def:sing}
Let $u:\Omega\to X$ be a minimizing map into an $F$-connected complex.  We say that $x_0\in\Omega$ is a {\bf regular point} if there is some $\sigma_0>0$, and some $k$-flat $F\subset X_{u(x_0)}$, so that $u(B_{\sigma_0}(x_0))\subset F$.

If $x_0$ is not a regular point of $u$, we refer to it as a {\bf singular point} of $u$.  We write $\Si(u)$ to denote the collection of all singular points of $u$, and we additionally write $\Siu$ to denote the collection of singular points of order strictly greater than $1$.

For convenience we write $\A$ to denote the {\bf set of points $x$ of high order}, that is, the $x$ such that $\text{Ord}_u(x)>1$.  In particular, in this notation $\Siu=\Si(u)\cap\A$.
\end{definition}

For our purposes, the results of \cite{GS} pertaining to $F$-connected complexes are summarized in the following theorem.

\begin{theorem}[Theorem 6.3 (i) and (ii) and parts of Theorem 6.4 of \cite{GS}]\label{thm:BigGS}
Suppose that $X$ is an $F$-connected complex.  Then the following hold:
\begin{enumerate}
\item For any positive integer $n$ and any compact $K_0\subset X$, there exists $\widehat{\epsilon}>0$ depending only on $K_0$ and $n$ so that if $u:\Omega^n\to X$ is a minimizing map from an $n$-dimensional domain, and $u(\Omega)\subset K_0$, then for all $x\in\Omega$ we have that either $\text{Ord}_u(x)=1$ or $\text{Ord}_u(x)\geq1+\widehat{\epsilon}$.
\item Let $u:\Omega\to X$ be a minimizing map, and let $x_0\in\Omega$ have $\text{Ord}_u(x_0)=1$.

Then, there exists a totally geodesic subcomplex $X_0$ of $X_{u(x_0)}$ which is isometric to $\R^m\times X_1^{k-m}$, where $X_1^{k-m}$ is an $F$-connected complex of dimension $k-m$, so that $u(B_{\sigma_0}(x_0))\subset X_0$ for some $\sigma_0>0$.  Here, $m$ is a positive integer $1\leq m\leq\min\{n,k\}$ where $n$ is the dimension of the domain manifold.

Moreover, if we write $u=(u_1,u_2):B_{\sigma_0}(x_0)\to\R^m\times X_1^{k-m}$, then $u_1$ is a harmonic map of rank $m$ at every point of $B_{\sigma_0}(x_0)$, and $\text{Ord}_{u_2}(x_0)>1$.
\item Let $u:\Omega\to X$ be a minimizing map from $\Omega$ of dimension $m$ into an $F$-connected complex.  Then the Hausdorff dimension of $\Si(u)$ is at most $m-2$.
\item Let $u:\Omega\to X$ be a minimizing map from $\Omega$ of dimension $m$ into an $F$-connected complex.  Then the Hausdorff dimension of $\A$ is at most $m-2$.
\end{enumerate}
\end{theorem}

We remark that the last item, (4), of the above theorem is not explicitly stated as a result of \cite{GS}.  However, it is proved as one step in the proof of (3).  This fact is particularly valuable for our work in this paper, so we feel that it is appropriate to state it as a result in its own right.  

It is also worth noting that if $u:\Omega\to X$ is a minimizing map, where $X$ is a $1$-dimensional $F$-connected complex, then (2) of the above theorem implies that any point $x_0$ of order $1$ is in fact a regular point.  This is because the decomposition of (2) here gives that $u$ takes a small ball about $x_0$ into a complex isometric with $\R^1\times X^0$, where $X^0$ is a zero-dimensional simplicial complex.  Hence $u$ in fact maps into a complex isometric to $\R$ near $x_0$, which is to say that $x_0$ is regular.  In particular, for such complexes we have that $\Si(u)=\Siu$.

\begin{lemma}\label{lem:1dSi}
If $u:\Omega\to X$ is a minimizing map into a $1$-dimensional $F$-connected complex, then $\Si(u)=\Siu$.
\end{lemma}

With this lemma and part (2) of the previous theorem, we are able to demonstrate that Theorem \ref{Thm1.2} implies Theorem \ref{Thm1.3}.  This proof is reminiscent of the proof of Theorem 2.4.(3) presented in \cite{GS}.  Before stating this proof, let us define rectifiability (although for our purposes at present it suffices to know that the countable union of $k$-rectifiable sets remains $k$-rectifiable).

\begin{definition}
Suppose that a subset $A\subset\R^m$ is of Hausdorff dimension $k$.  We then say that $A$ is {\bf countably $k$-rectifiable} if there is a countable collection of Lipschitz maps $\{f_i\}_{i=1}^\infty$, $f_i:\R^k\to\R^m$ such that $\bigcup f_i(\R^k)$ covers $A$ except for a set of $k$-dimensional Hausdorff measure $0$.
\end{definition}

\begin{proof}[Proof of Theorem \ref{Thm1.3}]
We induct on the dimension of the $F$-connected complex $X$.  When $d=1$, Lemma \ref{lem:1dSi} implies that $\Si(u)=\Siu$, and because Theorem \ref{Thm1.2} directly implies that $\Siu$ is countably $(m-2)$-rectifiable, we conclude the desired result when $d=1$.

For $d\geq2$, assume that the theorem holds for all $F$-connected complexes of dimension less than $d$.  Let $u:\Omega\to X$ be a minimizing map from $\Omega\subset\R^m$ to $X$, an $F$-connected complex of dimension $d$.  Since $\Siu$ is countably $(m-2)$-rectifiable by Theorem \ref{Thm1.2}, it suffices to show that $\Si(u)\setminus\Siu$ is countably $(m-2)$-rectifiable.  

Let $x_0\in\Si(u)\setminus\Siu$ be a singular point with $\text{Ord}_u(x)=1$.  Then, Theorem \ref{thm:BigGS} part (2) implies that there is a ball $B=B_r(x)$ so that $u|_B$ maps into a totally geodesic subcomplex $X_0\subset X$ which is isometric to $\R^k\times X_2$, where $X_2$ is an $F$-connected complex of dimension $d-k$, and $k\geq1$.  By the same result, we can consider $u=(u_1,u_2)$ as a map into this product complex where $u_1$ is smooth harmonic (in fact, analytic) and $u_2$ is energy minimizing.  Since $u_1$ is a regular map, the singular set of $u|_B$ coincides with the singular set of $(u_2)|_B$.

Since $u_2:B\to X_2$ is a map into an $F$-connected complex of dimension strictly less than $d$, by the inductive hypothesis $\Si(u_2)=\Si(u)\cap B$ is countably $(m-2)$-rectifiable.  Covering the singular set of $u$ by countably many balls $B_{\epsilon_i}(x_i)$, we conclude that $\Si(u)\setminus\Siu$ is countably $(m-2)$-rectifiable, and thus conclude that $\Si(u)$ is countably $(m-2)$-rectifiable.
\end{proof}

\subsection{Conical complexes}\label{subsec:cone}
In this paper, rather than working with general $F$-connected complexes, we will use a class of simpler complexes, the conical $F$-connected complexes.

Recall that for a nonpositively curved complex $X$, we have defined a family of maps $R_{P,\lambda}:X\to X$.  These maps (as $\lambda$ varies) are a homotopy from the constant map at $P$ to the identity on $X$, and $R_{P,\lambda}$ is Lipschitz with Lipschitz constant at most $\lambda$.  Writing this out more explicitly, to make the distinction in the next definition clearer, we have that for any $P,x,y\in X$,
\[
d(R_{P,\lambda}(x),R_{P,\lambda}(y))\leq\lambda d(x,y).
\]

\begin{definition}
We say that a nonpositively curved space $X$ is {\bf conical} with respect to the point $0_X$ if for all $0\leq\lambda\leq1$ and all $x,y\in X$ we have
\[
d(R_{0_X,\lambda}(x),R_{0_X,\lambda}(y))=\lambda d(x,y)
\]
and $R_{0_X,\lambda}$ is surjective for each $0<\lambda\leq1$.  In this case, we say that $0_X$ is a {\bf cone point} of the space.

Finally, suppose that for every $p\in X$, there is a neighborhood of $p$ isometric to a neighborhood of $0_Y$ in some conical space $Y$.  In this case, we call $X$ {\bf locally conical}.
\end{definition}

For example, tangent cones of $F$-connected complexes are conical with respect to $0$.  Importantly, for each point $p$ in an $F$-connected complex $X$, there is a small neighborhood of $p$ which is isometric to a small neighborhood of $0$ in the tangent cone.  In particular, all $F$-connected complexes are locally conical (in our terminology).

Euclidean space, of course, is conical with respect to any point.

We also remark at this point that if $0_X$ is a cone point of $X$, and $p$ is any point of $X$, there is a {\em unique} geodesic ray $\gamma_p$ with $\gamma_p(0)=0_X$ and $\gamma_p(1)=p$.  (This is in contrast to general NPC spaces, where in general there may be many such rays.)  This is because if $\gamma_1(t)$ and $\gamma_2(t)$ are two geodesic rays with $\gamma_i(0)=0$ and $\gamma_i(1)=p$, and $\omega>1$, then if we let $\lambda=\omega^{-1}$, we see that
\[
d(R_{0_X,\lambda}(\gamma_1(\omega)),R_{0_X,\lambda}(\gamma_2(\omega)))=d(\gamma_1(1),\gamma_2(1))=0
\]
by the definition of the retraction maps $R_{0_X,\lambda}$ and our choice of $\lambda$.  However, because $0_X$ is a cone point, we know in fact that
\[
0=d(R_{0_X,\lambda}(\gamma_1(\omega)),R_{0_X,\lambda}(\gamma_2(\omega)))=\lambda d(\gamma_1(\omega),\gamma_2(\omega))
\]
so that $\gamma_1(\omega)=\gamma_2(\omega)$ as well.  

Moreover, by the assumption that $R_{0_X,\lambda}$ is surjective for all $\lambda>0$, we also see that this ray is well-defined for all $\lambda\geq0$.  In particular, this surjectivity directly implies that for any $p\in X,\lambda\geq1$ there is some $q\in X$ and a geodesic ray $\gamma_q$ with $\gamma_q(0)=0_X$, $\gamma_q(1)=q$ and $\gamma_q(\lambda^{-1})=p$.  The uniqueness of the geodesic ray through $p$ implies that there is exactly one such $q$.

\begin{definition}
Throughout this paper, when $X$ is a conical nonpositively curved complex with cone point $0_X$, for any $p\in X$ and for any $\lambda\geq0$ we define the point $\lambda p$, the {\bf rescaling of $p$ about $0_X$ by $\lambda$}, by
\[
\lambda p=\gamma_p(\lambda)
\]
where $\gamma_p$ is a geodesic with $\gamma_p(0)=0_X$ and $\gamma_p(1)=p$.  The above remarks prove that this is well-defined.
\end{definition}

Because we usually work with conical complexes, we summarize a common assumption for convenience.

\begin{assumption}\label{cone} 
We assume that $u$ is a nonconstant minimizing map from $B_{64}(0)\subset\R^m$ into a conical $F$-connected complex $X$ with $u(0)=0_X$, a cone point of $X$.
\end{assumption}

\subsection{Smoothed Functionals}\label{subsec:smoothedfun}

Although we will sometimes make use of the unsmoothed energy, height, and frequency functions of Definitions \ref{def:energy} and \ref{def:height}, we most often work with the smoothed versions of these functionals, introduced in \cite{DS}.

\begin{definition}\label{def:smoothedfun}
Let $\phi:\R_{\geq0}\to\R$ be a nonincreasing Lipschitz function equal to $1$ on $\big[0,\frac{1}{2}\big]$ and equal to $0$ on $[1,\infty)$.

If $u:\Omega\to X$ is a nonconstant minimizing map from $\Omega\subset\R^m$ into a conical nonpositively curved space $X$, we define the {\bf smoothed energy} $E_\phi$, the {\bf smoothed height} $H_\phi$, and the {\bf smoothed frequency} $I_\phi$ of $u$ on the ball $B_r(x)$ to be

\begin{align}
E_\phi(x,r)&:=\int_{\Omega}|\nabla u|^2\phi\bigg(\frac{|y-x|}{r}\bigg)dy\\
H_\phi(x,r)&:=-\int_{\Omega} d^2(u(y),0_X)|y-x|^{-1}\phi'\bigg(\frac{|y-x|}{r}\bigg)dy\\
I_\phi(x,r)&:=\frac{rE_\phi(x,r)}{H_\phi(x,r)}.
\end{align}

We also define, for convenience,
\begin{equation}
\xi_\phi(x,r):=-\int_{\Omega}|\partial_{\nu_x}u(y)|^2|y-x|\phi'\bigg(\frac{|y-x|}{r}\bigg)dy,
\end{equation}
where $\nu_x(y)$ denotes the vector field $\nu_x(y):=\frac{y-x}{|y-x|}$.
\end{definition}

Notice that $H_\phi$ measures the height with respect to $0_X$.  It is not difficult to define a notion of the smoothed height where we measure the distance from other points in the complex, but we will not work with these.  Since the complexes here are conical, it is most natural and most convenient to work with respect to the cone point.  

We remark here that by Proposition \ref{prop:GS3.4}, if a minimizing map $u$ into a nonpositively curved space is constant on an open set, then it is constant everywhere.  Hence $H_\phi(x,r)$ is strictly positive whenever $r>0$ and $B_r(x)$ is contained in the domain of $u$, unless $u$ is identically equal to $0_X$.  Thus, the frequency function is also well-defined for all such $x,r$.

We have already mentioned the monotonicity formula for the unsmoothed frequency function; we will see that a monotonicity formula holds for the smoothed frequency function as well in Section \ref{Identities3}.

Similar to Assumption \ref{cone}, the following assumption is used often enough to merit summarization.

\begin{assumption}\label{bound} We fix some $\Lambda>0$ and assume that $I_\phi(0,64)\leq\Lambda$, $\phi'(t)=-2$ for $t\in(1/2,1)$ and $\phi'(t)=0$ otherwise.\end{assumption}

\subsection{Key Theorems}\label{subsec:key}

We now state two results which imply Theorem \ref{Thm1.2} (and hence Theorem \ref{Thm1.3} as well).  First, we recall the definition of the tubular neighborhood of a set.

\begin{definition}
For a set $E\subset\R^m$, the {\bf tubular neighborhood $B_r(E)$} of radius $r$ about $E$ is
\[
B_r(E):=\{y\in\R^m:d(x,y)<r\text{ for some }x\in E\}=\bigcup_{x\in E}B_r(x).
\]
\end{definition}

We then have the following results.

\begin{theorem}\label{MinkowskiBoundThm}
Under Assumptions \ref{cone} and \ref{bound}, there is a constant $C(m,X,\Lambda)$ so that
\begin{equation}
|B_\rho(\Siu\cap B_{1/8}(0))|\leq C\rho^2
\end{equation}
for all $\rho>0$.  In particular, in this case we have that $\mathcal{H}^{m-2}(\Siu\cap B_{1/8}(0))<\infty$.
\end{theorem}

\begin{theorem}\label{MicroRectThm}
Under Assumptions \ref{cone} and \ref{bound}, the set $\Siu\cap B_{1/8}(0)$ is countably $(m-2)$-rectifiable.
\end{theorem}

With these in hand, we recover Theorem \ref{Thm1.2} by a covering argument.

\begin{proof}[Proof of Theorem \ref{Thm1.2}]
Let $u:\Omega\to X$ be a minimizing map for some open domain $\Omega\subset\R^m$ and $X$ an $F$-connected complex.  For each $x\in\Omega$, because $X$ is locally conical, there is a neighborhood $V$ of $u(x)$ which is isometric to a neighborhood of the origin in the tangent cone of $X$ at $u(x)$.  Choose $r_x>0$ so that $B_{512r_x}(x)\subset\Omega$.  We may then consider $u|_{B_{512r_x}(x)}$ as a minimizing map from a ball to a conical $F$-connected complex, where $u(x)$ is a cone point of this complex (as in Assumption \ref{cone}).

The collection $\{B_{r_x}(x)\}_{x\in\Omega}$ is clearly a cover of $\Omega$.

Let $K\subset\Omega$ be compact.  We cover $K$ by finitely many of our $B_{r_{x_i}}(x_i)$.  Among this finite collection, we find the maximum value of $I_\phi(x_i,512r_{x_i})=:\Lambda$, where $\phi$ is as in Assumption \ref{bound}.  We apply a rescaled version of Theorem \ref{MinkowskiBoundThm} to conclude that for $B_i:=B_{r_{x_i}}(x_i)$,
\[
|B_\rho(\Siu\cap B_i)|\leq C \Big(\frac{1}{8r_{x_i}}\Big)^m \rho^2.
\]
Summing this estimate over all of the $B_i$, we obtain that
\[
|B_\rho(\Siu\cap K)|\leq C \rho^2
\]
where $C$ depends on the number of balls, on their radii, and on $\Lambda,m$, and $X$.  All of these are determined by $K$ and $u$, which completes the desired Minkowski-type bound of Theorem \ref{Thm1.2}.  It is then easy to see that $\mathcal{H}^{m-2}(\Siu\cap K)<\infty$.

Applying Theorem \ref{MicroRectThm} to each $B_i$, we see that $\Siu\cap B_i$ is countably $(m-2)$-rectifiable, and hence it is clear that $\Siu\cap K$ is as well.

Taking a sequence of compact sets $K_1\subset K_2\subset\dots$ so that $\bigcup_{i=1}^\infty K_i=\Omega$, we see that $\Siu$ is a countable union of countably $(m-2)$-rectifiable sets, and is hence countably $(m-2)$-rectifiable as well.
\end{proof}

\subsection{Sketch of the Paper's Organization}\label{subsec:sketch}
Before discussing the theorems of \cite{AT} and \cite{NV}, we must define the following quantity.

\begin{definition}
Let $\mu$ be a Radon measure on $\R^m$ and let and $k\in\{0,1,\dots,m-1\}$.  For $x\in\R^m$ and $r>0$, we define the {\bf $k^{\text{th}}$ mean flatness} of $\mu$ in the ball $B_r(x)$ to be
\[
D_\mu^k(x,r):=\inf_Lr^{-k-2}\int_{B_r(x)}\text{dist}(y,L)^2\phantom{i}d\mu(y),
\]
where the infimum is taken over all affine $k$-planes $L$.
\end{definition}

For our purposes, this is related to rectifiability by the following theorem of \cite{AT}.  In this, $\mathcal{H}^k$ denotes the $k$-dimensional Hausdorff measure on $\R^m$.

\begin{theorem}[\cite{AT}, Corollary 1.3]\label{thm:AT}
Let $S\subset\R^m$ be $\mathcal{H}^k$-measurable with $\mathcal{H}^k(S)<\infty$ and consider $\mu=\mathcal{H}^k\llcorner S$.  Then $S$ is countably $k$-rectifiable if and only if
\[
\int_0^1D_\mu^k(x,s)\frac{ds}{s}<\infty\hspace{5mm}\text{for }\mu\text{-a.e. }x.
\]
\end{theorem}

The rough intuition here is that the mean flatness detects how close a measure is to being supported on a $k$-plane, while rectifiable sets must have tangent planes almost everywhere.  Hence, as we ``zoom in" on almost any point $x$ of a rectifiable set, we expect the set, at small scales, to lie close to the tangent plane at $x$.  We use this result in Section \ref{Rectifiability8} to finally show that $\Siu$ is rectifiable, proving Theorem \ref{MicroRectThm}.

The next result, of \cite{NV}, is used in the covering arguments of Section \ref{Minkowski7}.  These arguments allow us to conclude that $\Siu$ has finite $m-2$ dimensional Hausdorff measure, and to prove Theorem \ref{MinkowskiBoundThm}.

\begin{theorem}[\cite{NV}, Theorem 3.4]\label{thm:NV}
Fix $k\leq m\in\N$, let $\{B_{s_j}(x_j)\}_{j\in J}\subseteq B_2(0)\subset\R^m$ be a sequence of pairwise disjoint balls centered in $B_1(0)$, and let $\mu$ be the measure
\[
\mu=\sum_{j\in J}s_j^k\delta_{x_j}.
\]

Then, there exist constants $\delta_0=\delta_0(m)$ and $C_R=C_R(m)$ depending only on $m$ such that if for all $B_r(x)\subset B_2(0)$ with $x\in B_1(0)$ we have the integral bound
\[
\int_{B_r(x)}\bigg(\int_0^rD_{\mu}^k(y,s)\frac{ds}{s}\bigg)d\mu(y)<\delta_0^2r^k
\]
then the measure $\mu$ is bounded by
\[
\mu(B_1(0))=\sum_{j\in J}s_j^k\leq C_R.
\]
\end{theorem}

These theorems motivate the main purpose of this paper's computations and lemmas, as they both require us to bound the mean flatness of certain measures.  

Our main tool for this purpose is the smoothed frequency function $I_\phi$.  For energy minimizers, this is a nondecreasing function of $r$; a fact we prove in Section \ref{Identities3} using the variational formulae of Section \ref{VarFormulae3}.  Moreover, for a map $u$ into a conical nonpositively curved space, $I_\phi$ detects how close $u$ is to a homogeneous map (cf. Definition \ref{def:ConHom}).

In particular, in Section \ref{Pinching4} we show that $u$ is homogeneous of degree $\alpha$ about the point $x$ if, and only if, $I_\phi(x,r_1)=\alpha=I_\phi(x,r_2)$ for $r_1<r_2$ (cf. Lemma \ref{homchar}).  This is a motivation for the following definition.

\begin{definition}\label{def:freqpinch}
For $u$ and $\phi$ as in Assumptions \ref{cone} and \ref{bound}, and for $x\in B_1(0)$, $0<s
\leq r$ we define
\begin{equation}
W_s^r(x):=I_\phi(x,r)-I_\phi(x,s)
\end{equation}
to be the {\bf frequency pinching} of $u$ between the radii $s$ and $r$.
\end{definition}

Section \ref{Pinching4} shows that, in a sense, the frequency pinching quantitatively controls how far a map is from being homogeneous.  In Section \ref{Jones5} we use these quantitative bounds to show the following result (which also appears as Proposition \ref{flatbound}).

\begin{proposition}
Under Assumptions \ref{cone} and \ref{bound}, there exists a positive constant $C(\Lambda,m,X)$ so that, for a finite nonnegative Radon measure $\mu$, and a ball $B_{r/8}(x_0)$ such that $B_{r/8}(x_0)\cap\A$ is nonempty, then
\[
D_\mu^{m-2}(x_0,r/8)\leq\frac{C}{r^{m-2}}\int_{B_{r/8}(x_0)}W_{r/8}^{4r}(x)\phantom{i}d\mu(x).
\]
\end{proposition}

This provides a bound for the mean flatness in terms of the frequency pinching, which will allow us  to apply Theorems \ref{thm:AT} and \ref{thm:NV}.

Section \ref{Qualitative6} proceeds by comparing maps to homogeneous maps in a qualitative sense.  In particular, we show that if $u$ has small frequency pinching at points spanning an $(m-2)$-dimensional affine subspace $L$, then $\A$ lies close to $L$, and that the frequency function is almost constant along $L$.  



In Section \ref{Minkowski7} we combine the bound of Proposition \ref{flatbound} with the results of Section \ref{Qualitative6} to prove Theorem \ref{MinkowskiBoundThm} using a covering lemma.  We then use Proposition \ref{flatbound} again in Section \ref{Rectifiability8} to bound the mean flatness of the $(m-2)$-dimensional Hausdorff measure on $\Siu$, and thus we apply Theorem \ref{thm:AT} to prove Theorem \ref{MicroRectThm}.

\section{Variational Formulae in Conical $F$-connected complexes}\label{VarFormulae3}
In this section, we work with a conical $F$-connected complex $X$, with fixed cone point $0_X$.  We consider this complex to be isometrically embedded in a Euclidean space $\R^N$ in the sense that the Riemannian metric induced on each simplex of $X$ by the embedding agrees with the metric on that simplex (as in \cite{GS}).  We further assume that $0_X$ is taken to $0$, and that the image of the embedding is a cone in $\R^N$.  In particular, we take the embedding so that geodesic rays from $0_X$ are taken to rays from $0$ in $\R^N$.  (This follows the framework of \cite{GS}, especially of their section 3.  Other geodesic rays in $X$ will not, in general, be taken to lines in $\R^N$, but to piecewise linear segments.)  For vectors $v,w$ in $\R^N$, we denote the inner product between them by $\la v,w\ra$.  Here, we will always take $v,w$ to be tangent to $X$.

We shall use the notation $|v|$ to denote the Euclidean length of a vector, and the notation $d(x,y)$ to denote the distance between $x,y$ in the $F$-connected complex $X$.  Because geodesics from $0_X$ are taken to rays in $\R^N$, for $x\in X$ we have $|x|=d(x,0_X)$.

Using this embedding allows us to more easily compute the derivative of the distance-squared function on $X$, $d^2(0_X,x)$, because it agrees with the Euclidean distance $|x|^2$.  One may reasonably object that this treatment is not intrinsic to $X$, but as long as the embedding is chosen to be isometric in the sense used here, with $0_X$ mapped to $0$, the derivative of the distance-squared, on $X$, is well-defined.  Indeed, this gradient can be described at a point $p\in X$ in a purely intrinsic fashion by regarding it as twice the velocity of the unique geodesic $\gamma_p:\R_{\geq0}\to X$ with $\gamma_p(0)=0_X$ and $\gamma_p(1)=p$.  (This is precisely analogous to the situation in Euclidean space, of course.)

\begin{convention}
Because there are a number of distinct operations in this section (and paper more generally) that can be regarded as taking products, we introduce some slightly nonstandard notation to disambiguate these.

\begin{enumerate}
\item For two vectors $v,w$ tangent to $\Omega$ at a point $x$, (i.e. the vectors lie in $\R^m$), we write their inner product as $v\cdot w$.
\item For two vectors $v,w$ tangent to $X$ at a point $p$, we write their inner product as $\la v,w\ra$.
\item We write $\nabla f \circ \nabla g$ to denote $\nabla(f\circ g)$, expanded using the chain rule.
\item Finally, we occasionally write $\nabla f \circ v$ to denote $\partial_v f$ for a vector $v$ tangent to the domain of $f$.  We do this when we wish to emphasize that the directional derivative may be computed by evaluating the linear operator $\nabla f$ at the vector $v$. 
\end{enumerate}


(Note that $\nabla f\circ\nabla g$ is the same as the ordinary matrix multiplication of the differentials $\nabla f$ and $\nabla g$ in this setting; we simply wish to emphasize the distinction between this and the inner product $\nabla f\cdot \nabla g$, which makes sense only if $f,g$ are real-valued functions sharing $\Omega$ as their source.)
\end{convention}



The two principal variational identities we will use in the next section are given by the following two lemmas.

\begin{lemma}\label{extvar}
For any minimizing map $u:\Omega\to X$ (with image lying in a bounded subset of $X$), where $X$ is a conical $F$-connected complex with cone point $0_X$, the following holds in the weak sense:
\[
\Delta d^2(u,0_X)=2|\nabla u|^2.
\]

In other words, for every smooth, compactly supported $\phi:\Omega\to\R$, 
\[
2\int_\Omega\phi(x)|\nabla u|^2(x)dx+\int_\Omega\nabla\phi\cdot\nabla d^2(u(x),0_X)dx=0.
\]
\end{lemma}

This lemma is essentially the differential inequality of Proposition 2.2 of \cite{GS}, adapted to the case of a conical target space.  In this context, rather than a differential inequality, we are able to prove that this is an identity; this is a major benefit of working with a conical target.

\begin{lemma}\label{intvar}
For any smooth compactly supported $\phi:\Omega\to\R^m$ (where $\Omega$ is a bounded domain of $\R^m$) and any minimizing map $u:\Omega\to X$ into $X$ a conical $F$-connected complex, we have that
\[
2\int_\Omega\la \nabla u,\nabla u\circ\nabla \phi\ra dy-\int_\Omega|\nabla u|^2\text{div}(\phi)dy=0.
\]
\end{lemma}

We prove the former lemma by considering variations of the map on the target space $X$, and the latter by considering variations on the domain $\Omega$.  

\begin{proof}[Proof of Lemma \ref{extvar}]
The computations here are based around the scaling maps $R_{0_X,\lambda}(p)=:\lambda p$, in the notation introduced in the previous section.  Note that because $X$ is conical, these maps are defined for all $\lambda\geq0$ (rather than only for $0\leq\lambda\leq1$) and that we have the equality
\[
d(R_{0_X,\lambda}(p),R_{0_X,\lambda}(q))=\lambda d(p,q).
\]

We define $u_t(x)=R_{0_X,1+t\varphi(x)}(u(x))$ for some smooth $\varphi(x):\Omega\to\R$.  Its partial derivatives satisfy 
\begin{align*}
\bigg|\frac{\partial u_t}{\partial x_i}(x)\bigg|^2=|(\partial_{\partial u/\partial x_i}&R_{0_X,1+t\varphi(x)})(u(x))|^2+t\frac{\partial\varphi}{\partial x_i}\frac{\partial d^2(R_{0_X,1+t\varphi(x)}(u(x),0_X)}{\partial x_i}\\&+t^2\bigg(\frac{\partial\varphi}{\partial x_i}\bigg)^2\bigg|\frac{\partial R_{0_X,1+t\varphi(x)}}{\partial\lambda}(u(x))\bigg|^2
\end{align*}
where $\frac{\partial}{\partial\lambda}$ denotes the derivative of $R_{0_X,\lambda}(p)$ with respect to $\lambda$.  We omit the computations here, but refer the interested reader to \cite{GS}, where this equality is justified leading up to Proposition 2.2, which is the analogous result to our Lemma \ref{extvar} when the target space is not required to be conical.

We now note that the expression
\[
\bigg(\frac{\partial\varphi}{\partial x_i}\bigg)^2\bigg|\frac{\partial R_{0_X,1+t\varphi(x)}}{\partial\lambda}(u(x))\bigg|^2
\]
is uniformly bounded, as $\big|\frac{\partial R_{0_X,\lambda}}{\partial\lambda}(p)\big|=\big|\frac{\partial}{\partial\lambda}\gamma_p(\lambda)\big|=d(p,0_X)$, so our assumptions on $u,\varphi$ bound this term (uniformly for all $t$ small enough).  

Moreover, the term
\[
\frac{\partial\varphi}{\partial x_i}\frac{\partial d^2(R_{0_X,1+t\varphi(x)}(u(x),0_X)}{\partial x_i}
\]
is continuous in $t$, by inspection, because
\[
d^2(R_{0_X,1+t\varphi(x)}(u(x)),0_X)=(1+t\varphi(x))^2d^2(u(x),0_X)
\]
and we can simply take the derivative of the right-hand side in $x$.  This coefficient is bounded by the energy of $u$:
\[
\nabla(d^2(R_{0_X,1+t\varphi(x)}(u(x)),0_X))=\nabla(1+t\varphi(x))^2 d^2(u(x),0_X)+(1+t\varphi(x))^2\nabla d^2(u(x),0_X)
\]
so that
\[
|\nabla(d^2(R_{0_X,1+t\varphi(x)}(u(x)),0_X))|\leq C_1+C_2|\nabla u|^2
\]
where $C_1,C_2$ are constants depending only on $\varphi$ and the upper bound on $d(u(x),0_X)$.  We note that this right-hand side is integrable on $\Omega$, so the hypotheses of the dominated convergence theorem are satisfied.

We then find the energy of $u_t$, using the identity $d(R_{0_X,\lambda}(p),R_{0_X,\lambda}(q))=\lambda d(p,q)$ to compute the first term explicitly:
\begin{align*}
E(u_t)=\int_{\Omega}(1+t\varphi(x))^2|\nabla u|^2dx+t\int_\Omega \nabla\varphi\cdot \nabla d^2(R_{0_X,1+t\varphi}(u),0_X)dx+O(t^2).
\end{align*}

Our earlier computations let us take the derivative of $E(u_t)$ with respect to $t$ at $t=0$.  By the dominated convergence theorem, the integral \[\int_\Omega \nabla\varphi\cdot \nabla d^2(R_{1+t\varphi}(u),0_X)dx\] is continuous (as a function of $t$) at $t=0$, and thus the derivative of \[t\int_\Omega \nabla\varphi\cdot \nabla d^2(R_{1+t\varphi}(u),0_X)dx\] at $t=0$ is $\int_\Omega \nabla\varphi\cdot \nabla d^2(u,0_X)dx$.  (The derivative of the $O(t^2)$ term at $t=0$ is, of course, $0$.)

We thus see directly that $E(u_t)$ is differentiable in $t$ at $t=0$.  Since $E(u_t)$ has a local minimum there, this derivative must be $0$, so we conclude that

\[
0=\frac{d}{dt}\bigg|_{t=0}E(u_t)=2\int_{\Omega}\varphi(x)|\nabla u|^2dx+\int_\Omega \nabla\varphi\cdot \nabla d^2(u,0_X)dx,
\]
exactly as we wished to show.
\end{proof}

The fact that this identity is an equality in the conical setting is the main reason for our focus on this setting; in more general spaces of nonpositive curvature, it would be an inequality which would introduce error terms to our later computations.

\begin{proof}[Proof of Lemma \ref{intvar}]
Let $\varphi:\Omega\to\R^m$ be a smooth compactly supported function on $\Omega\subset\R^m$ and consider $\Phi_t(x)=x+t\varphi(x)$; notice that for sufficiently small $t$, $\Phi_t$ is a diffeomorphism of $\Omega$ onto itself.  We then define $u_t=u(\Phi_t)$ and see that
\[
E(u_t)=\int_\Omega|\nabla u_t(x)|^2=\int_\Omega|\nabla u(x+t\varphi(x))+t\nabla u(x+t\varphi(x))\circ \nabla\varphi(x)|^2 dx.
\]

Changing variables to $y=\Phi_t(x)$, we compute
\[
E(u_t)=\int_\Omega|\nabla u(y)+t\nabla u(y)\circ \nabla\varphi(\Phi^{-1}_t(y))|^2\det(\nabla\Phi^{-1}_t)dy.
\]
If we expand the determinant here as a polynomial in $t$ and expand the inner product, we see that all of the resulting terms are bounded in terms of the energy density (and constants depending only on $\varphi$).  Moreover, we can see that
\[
E(u_t)=E(u)+t\underbrace{\bigg(2\int_\Omega\la \nabla u(y),\nabla u(y)\circ \nabla\varphi(\Phi^{-1}_t(y))\ra dy-\int_\Omega|\nabla u|^2\text{div}(\varphi)dy\bigg)}_{=:F_1(t)}+o(t).
\]

Just as in the previous proof, if $\lim_{t\to0}F_1(t)$ exists, then $E(u_t)$ is differentiable at $t=0$, and its derivative at $t=0$ is equal to this limit.  But this limit exists by the dominated convergence theorem, since the integrands converge pointwise.

Hence,
\[
0=\frac{d}{dt}\bigg|_{t=0}E(u_t)=2\int_\Omega\la \nabla u,\nabla u\circ \nabla\varphi\ra dy-\int_\Omega|\nabla u|^2\text{div}(\varphi)dy.
\]
\end{proof}

\section{Useful Identities and Inequalities}\label{Identities3}
The previous section's variational formulae let us prove a number of useful identities; the bulk of our subsequent computations rely in some way on the following identities.  

We begin by collecting a number of identities on the functions $E_\phi,H_\phi,I_\phi$ and their derivatives.

In these identities, we will use the radial vector field $\nu_x=\frac{y-x}{|y-x|}$ and the auxiliary function $\xi_\phi(x,r)$, which is given by\[
\xi_\phi(x,r):=-\int_{\R^m}|\partial_{\nu_x}u(y)|^2|y-x|\phi'\bigg(\frac{|y-x|}{r}\bigg)dy.
\]

\begin{proposition}\label{calculations}
We have that the functions $E_\phi,$ $H_\phi$, and $I_\phi$ are $C^1$ in both variables, and the following identities hold:
\begin{align}
E_\phi(x,r)&=-\frac{1}{r}\int\phi'\bigg(\frac{|y-x|}{r}\bigg)\la u(y),\partial_{\nu_x} u(y)\ra dy\label{calc1}\\
\partial_rE_\phi(x,r)&=\frac{m-2}{r}E_\phi(x,r)+\frac{2}{r^2}\xi_\phi(x,r)\label{calc2}\\
\partial_vE_\phi(x,r)&=-\frac{2}{r}\int\phi'\bigg(\frac{|y-x|}{r}\bigg)\la\partial_{\nu_x}u(y),\partial_vu(y)\ra dy\label{calc3}\\
\partial_rH_\phi(x,r)&=\frac{m-1}{r}H_\phi(x,r)+2E_\phi(x,r)\label{calc4}\\
\partial_vH_\phi(x,r)&=-2\int\phi'\bigg(\frac{|y-x|}{r}\bigg)\la u(y),\partial_{v}u(y)\ra dy\label{calc5}\end{align}
Further, both $r^{1-m}H_\phi(x,r)$ and $I_\phi(x,r)$ are nondecreasing as functions of $r$ and in particular:
\begin{align}
\partial_rI_\phi(x,r)&=\frac{2}{rH_\phi(x,r)^2}\bigg(H_\phi(x,r)\xi_\phi(x,r)-r^2E_\phi(x,r)^2\bigg)\geq0\label{calc6}\\
s^{1-m}H_\phi(x,s)&=r^{1-m}H_\phi(x,r)\exp\bigg(-2\int_s^rI_\phi(x,t)\frac{dt}{t}\bigg)\label{calc7}.
\end{align}
\end{proposition}
\begin{proof}
We will prove these formulae under the assumption that $\phi$ is smooth, as we can then recover the result for Lipschitz $\phi$ by taking a sequence of $\phi_k$ which are bounded in $C^1$ and converge to the desired $\phi$ in $W^{1,p}$ for $p<\infty$.  Then $E_\phi,H_\phi$, and all of their derivatives converge uniformly to the limits, and since $H_\phi$ is strictly positive for $r>0$, we conclude that $I_\phi$ is $C^1$ as well.

First, (\ref{calc1}) is a consequence of Lemma \ref{extvar}, applied to $\phi\big(\frac{|y-x|}{r}\big)$.  We use the fact that $\nabla d^2(x,0_X)=2x$ so that 

\begin{equation}\partial_{\nu_x}d^2(u(y),0_X)=2\la u(y),\partial_{\nu_x}u(y)\ra.\label{radDeriv}\end{equation}

Then, because $\nabla\big(\phi\big(\frac{|y-x|}{r}\big)\big)=\phi'\big(\frac{|y-x|}{r}\big)\partial_{\nu_x}$, we see that
\begin{align*}
\nabla\Big(\phi\Big(\frac{|y-x|}{r}\Big)\Big)\cdot\nabla(d^2(u(y),0_X))&=\phi'\Big(\frac{|y-x|}{r}\Big)\partial_{\nu_x}\cdot\nabla d^2(u(y),0_X)\\
&=\phi'\Big(\frac{|y-x|}{r}\Big)\partial_{\nu_x}d^2(u(y),0_X)
\end{align*}
from which we apply (\ref{radDeriv}) to derive (\ref{calc1}).

Differentiating the energy in $r$ and $x$, we see that
\begin{align*}
\partial_rE_\phi(x,r)&=-\int|\nabla u(y)|^2\phi'\bigg(\frac{|y-x|}{r}\bigg)\frac{|y-x|}{r^2}dy\\
\partial_vE_\phi(x,r)&=-\int|\nabla u(y)|^2\phi'\bigg(\frac{|y-x|}{r}\bigg)\frac{y-x}{r|y-x|}\cdot vdy.
\end{align*}

Applying Lemma \ref{intvar} with $\varphi_1(y)=\phi\big(\frac{|y-x|}{r}\big)(y-x)$ and $\varphi_2(y)=\phi\big(\frac{|y-x|}{r}\big)v$, respectively, we conclude (\ref{calc2}) and (\ref{calc3}).

We then change variables in the height function in two ways:
\begin{align}
H_\phi(x,r)&=-\int d^2(u(x+z),0_X)|z|^{-1}\phi'\bigg(\frac{|z|}{r}\bigg)dz\label{zint}\\&=-\frac{1}{r^{m-1}}\int d^2(u(x+r\zeta),0_X)|\zeta|^{-1}\phi'(|\zeta|)d\zeta.\label{rint}
\end{align}

Using (\ref{rint}) and differentiating in $r$ we see that
\[
\partial_rH_\phi(x,r)=\frac{m-1}{r}H_\phi(x,r)-\frac{2}{r^{m-1}}\int \la u(x+r\zeta),\partial_{r}(u(x+r\zeta))\ra|\zeta|^{-1}\phi'(|\zeta|)d\zeta
\]
where we have again used (\ref{radDeriv}).  Changing variables back to $y$ and applying (\ref{calc1}), we conclude that (\ref{calc4}) holds.

Using (\ref{zint}), we deduce (\ref{calc5}) by differentiating under the integral sign.

We compute $\partial_rI_\phi(x,r)$ by the quotient rule.  This derivative is nonnegative by Cauchy-Schwarz:
\begin{align*}
r^2E_\phi(x,r)^2&=\Bigg(\int-\phi'\bigg(\frac{|y-x|}{r}\bigg)\la u(y),\partial_{\nu_x}u(y)\ra dy\bigg)^2\\
&\leq\int-\phi'\bigg(\frac{|y-x|}{r}\bigg)|y-x|^{-1}|u(y)|^2dy\int-\phi'\bigg(\frac{|y-x|}{r}\bigg)|y-x||\partial_{\nu_x}u(y)|^2dy\\
&=H_\phi(x,r)\xi_\phi(x,r).
\end{align*}

We remark that we only use the assumption that $\phi'\leq0$ in this application of the Cauchy-Schwarz inequality.

Lastly, we rewrite (\ref{calc4}) as
\[
\partial_r\log(r^{1-m}H_\phi(x,r))=\frac{\partial_rH_\phi(x,r)}{H_\phi(x,r)}-\frac{m-1}{r}=2\frac{E_\phi(x,r)}{H_\phi(x,r)}=\frac{2}{r}I_\phi(x,r)
\]
and conclude (\ref{calc7}) immediately.
\end{proof}

If we let $\phi$ increase to the indicator function of $[0,1]$, we obtain similar statements for the classical variants of these functions; in particular we can derive the scale-invariant monotonicity of the height from (\ref{calc7}).  That is, for $0<s\leq r<\text{dist}(x,\partial\Omega)$, if $H$ denotes the classical height, we have
\begin{equation}
s^{1-m}H(x,s)\leq r^{1-m}H(x,r).\label{scinvmono}
\end{equation}

(Following the computations of \cite{GS} for minimizing maps with Euclidean domains will recover this monotonicity as well; cf. their proof of Theorem 2.3.)

The next lemma shows that the height (resp. frequency), measured on some ball (about a point $x$, with radius $r$), gives an upper bound on the same quantity at nearby points and smaller scales.

\begin{lemma}\label{lemma3.4}
There is a constant $C(m,\phi)$ such that
\begin{align}
H_\phi(y,\rho)&\leq CH_\phi(x,4\rho) &\forall y\in B_\rho(x)\subset B_{4\rho}(x)\subset\Omega\label{3.2hbound}\\
I_\phi(y,r)&\leq C(I_\phi(x,16r)+1) &\forall y\in B_{r/4}(x)\subset B_{16r}(x)\subset\Omega.\label{3.2fbound}
\end{align}
\end{lemma}
\begin{proof}
Without loss of generality take $x=0$ and assume $\rho=1$.  We shall prove (\ref{3.2hbound}).  The scale-invariant monotonicity of height (\ref{scinvmono}) easily implies that for $r\in(2,4)$,
\[
\int_{B_2}d^2(u(x),0_X)\leq C\int_{\partial B_r}d^2(u(x),0_X).
\]

Integrating both sides against $-r^{-1}\phi'(r/4)dr$, we then see that
\[
\int_{B_2}d^2(u(x),0_X)\leq CH_\phi(0,4)
\]
and since $B_1(y)\subset B_2(0)$, we clearly have that $H_\phi(y,1)\leq C\int_{B_2}d^2(u(x),0_X)$.  This proves (\ref{3.2hbound}).

Next, to show (\ref{3.2fbound}) we take $r=1$ and $x=0$, and then use (\ref{3.2hbound}), (\ref{calc7}), and the monotonicity of the frequency to conclude that
\begin{align*}
H_\phi(y,4)&\leq CH_\phi(0,16)\leq Ce^{CI_\phi(0,16)}H_\phi(0,1/4)\leq Ce^{CI_\phi(0,16)}H_\phi(y,1)\\
&=CH_\phi(y,4)\exp\bigg(CI_\phi(0,16)-2\int_1^4I_\phi(y,t)\frac{dt}{t}\bigg).
\end{align*}
Then, dividing by $H_\phi(y,4)$ and taking a logarithm, we conclude that
\[
2I_\phi(y,1)\int_1^4\frac{dt}{t}\leq C(1+I_\phi(0,16)).
\]
\end{proof}

\section{The Frequency Pinching}\label{Pinching4}

We wish to compare our minimizing maps to homogeneous minimizing maps.  The key tool is the frequency pinching, as discussed in Definition \ref{def:freqpinch} of Subsection \ref{subsec:sketch},
\[
W_s^r(x):=I_\phi(x,r)-I_\phi(x,s).
\]

This quantity is notable because it measures how far a map is from being homogeneous about $x$.  In particular, a map is homogeneous about $x$ if and only if the pinching at $x$ is zero.

Here, we wish to consider functions homogeneous with respect to the cone point $0_X$, rather than the ``intrinsic homogeneity" of \cite{GS}.

\begin{definition}\label{def:ConHom}
We say that a map $u:B_r(x_0)\to X$ is {\bf conically homogeneous about $x_0$ of order $\alpha$} if for any $x\in B_r(x_0)\setminus\{x_0\}$, we have that
\[
u(x)=|x-x_0|^\alpha u\Big(\frac{r(x-x_0)}{|x-x_0|}+x_0\Big)
\]
where the multiplication on the right-hand-side of the above identity denotes the scaling about the cone point.

For convenience, we usually refer to conically homogeneous maps as homogeneous; we will not be working with any other kind of homogeneity in this paper.

We will also say that $u$ is {\bf homogeneous about $x_0$ on the annulus} \[A_{s_1}^{s_2}(x_0):=\{x\in B_r(x_0):s_1\leq|x-x_0|\leq s_2\}\]
if, for any $x\in A_{s_1}^{s_2}(x_0)$, we have that
\[
u(x)=\frac{|x-x_0|^\alpha}{s_2^\alpha}u\Big(\frac{s_2(x-x_0)}{|x-x_0|}+x_0\Big).
\]

A map which is homogeneous on the ball $B_r(x_0)$ is homogeneous on any annulus $A_{s_1}^{s_2}(x_0)$ with $s_1\leq s_2\leq r$.  Additionally, if $u$ is homogeneous of degree $\alpha$ on the two annuli $A_{s_1}^{s_2}(x_0)$ and $A_{s_2}^{s_3}(x_0)$ (or more generally two annuli which overlap), it is homogeneous of degree $\alpha$ on their union $A_{s_1}^{s_3}(x_0)$.
\end{definition}

We recall (\ref{calc1}), which states that
\[
E_\phi(x,r)=-\frac{1}{r}\int\phi'\Big(\frac{|y-x|}{r}\Big)\partial_ru(y)\cdot u(y)dy.
\]

Applying this to a conically homogeneous map, we see that $rE_\phi(x_0,r)=\alpha H_\phi(x_0,r)$ by explicitly computing the radial derivative.  Thus, for any $r>0$, we have that $I_\phi(x_0,r)=\alpha$.

The following is a converse to the above remark.
\begin{lemma}\label{homchar}
If $u$ is a harmonic map $u:\Omega\to X$, where $\Omega\subset\R^m$ and $X$ is a conical $F$-connected complex, and $I_\phi(x,s_1)=I_\phi(x,s_2)$ for $s_1<s_2$, then $u$ is conically homogeneous about $x$ on $B_{s_2}(x)$.
\end{lemma}

In particular, a map is homogeneous about a point $x_0$ if and only if the order at that point is constant between two radii $s<r$.  Although we prove this lemma in this section, we mostly do so for motivation, as we do not require the result until Section \ref{Qualitative6}.  However, this lemma lets us consider the rigid case of the following claims, which show how the pinching at $x_0$ quantitatively measures how close a map is to being homogeneous about $x_0$.

\begin{theorem}\label{thm4.2}
There exists a $C=C(\Lambda,m,X)>0$ so that if $u,\phi$ satisfy Assumptions \ref{cone} and \ref{bound}, $x_1,x_2\in B_{1/8}(0)$, and $|x_1-x_2|\leq r/4$, then
\begin{equation}
|I_\phi(z,r)-I_\phi(y,r)|\leq C\bigg[W_{r/8}^{4r}(x_1)+W_{r/8}^{4r}(x_2)\bigg]^{\frac{1}{2}}\frac{|z-y|}{|x_1-x_2|}
\end{equation}
for $z,y$ on the line segment between $x_1$ and $x_2$.
\end{theorem}

In the rigid case, when the pinchings on the right are both $0$, $u$ is homogeneous about both $x_1$ and $x_2$, and is (as we shall show in Lemma \ref{twohomogeneous}) invariant along the line between $x_1$ and $x_2$.  Hence, the frequency $I_\phi(x,r)$ is also invariant along the line between $x_1$ and $x_2$.  In particular, the left-hand side of this inequality is $0$ when the right-hand side is.  

We can view this theorem as an approximate version of the above observation, stating that if a map is almost homogeneous about two distinct points, then its frequency function (at a fixed radius) is almost invariant along the line between them.

\begin{proposition}\label{prop4.3}
There exists a $C=C(\Lambda,m,X)$ so that if $u,\phi$ satisfy Assumptions \ref{cone} and \ref{bound}, for every $x\in B_{1/8}(0)$,
\begin{equation}\label{eq4.3}
\int_{B_2(x)\setminus B_{1/4}(x)}|\nabla u(z)\circ(z-x)-I_\phi(x,|z-x|)u(z)|^2\leq CW_{1/8}^4(x).
\end{equation}
\end{proposition}

Again, we compare to the rigid case for motivation.  If $u$ is homogeneous of order $\alpha$ about a point $x$, we can compute the derivative $\nabla u(z)\circ(z-x)=\alpha u(z)$ using the definition of conical homogeneity, and $\alpha=I_\phi(x,r)$ for any $r$ (and in particular for $r=|z-x|$).  This proposition states that the pinching bounds the $L^2$ difference of these quantities on a specific annulus; in the rigid case when the pinching vanishes, this $L^2$ difference vanishes as well.


\begin{proof}[Proof of Lemma \ref{homchar}]
For the sake of simplicity, take $x_0=0$ and suppose that $I_\phi(0,s_1)=I_\phi(0,s_2)=\alpha$.  Under these hypotheses, we have that for $s_1<r<s_2$, $\partial_rI_\phi(0,r)=0$ and that $I_\phi(0,r)\equiv\alpha$ for such $r$.  Then we must have
\[
r^2E_\phi(0,r)^2=H_\phi(0,r)\xi_\phi(0,r).
\]
On the other hand, if we expand the energy by (\ref{calc1}), we see that
\begin{align*}
r^2E_\phi(0,r)^2&=\bigg(\int-\phi'\Big(\frac{|y|}{r}\Big)\la\partial_{\nu_x}u(y), u(y)\ra dy\bigg)^2\\
&\leq\int-\phi'\Big(\frac{|y|}{r}\Big)|y|^{-1}|u(y)|^2dy\int-\phi'\Big(\frac{|y|}{r}\Big)|y|\partial_{\nu_0}u(y)|^2dy\\
&=H_\phi(0,r)\xi_\phi(0,r)
\end{align*}
where we have applied Cauchy-Schwarz to obtain the inequality in the middle step.  Since we have equality in Cauchy-Schwarz, we must have that almost everywhere on the support of $\phi'\big(\frac{|y|}{r}\big)$, we have that $|y|^{-1/2}u(y)$ and $|y|^{1/2}\partial_{\nu_0}u(y)$ are a fixed scalar multiple of each other.  In particular, the support of $\phi'\big(\frac{|y|}{r}\big)$ is the annulus $A_{r/2}^{r}(0)=\{y:\frac{r}{2}\leq|y|\leq r\}$.  If we write $\partial_{\nu_0}u(y)=\frac{\beta}{|y|}u(y)$, we can directly see that
\begin{align*}
rE_\phi(0,r)&=\int-\phi'\Big(\frac{|y|}{r}\Big)\la\partial_{\nu_0}u(y), u(y)\ra dy\\
&=\int-\phi'\Big(\frac{|y|}{r}\Big)\frac{\beta}{|y|}|u(y)|^2dy = \beta H_\phi(0,r).
\end{align*}

But we know, from the fact that $I_\phi(0,r)=\alpha$, that $rE_\phi(0,r)=\alpha H_\phi(0,r)$.  Hence 

\begin{equation}\label{eq:homchar1}
\partial_{\nu_0}u(y)=\frac{\alpha}{|y|}u(y)
\end{equation}
when $\frac{r}{2}\leq|y|\leq r$ .  Allowing $r$ to vary between $s_1,s_2$ we see that this holds for $\frac{s_1}{2}\leq|y|\leq s_2$.  Due to this, we see that $u(y)$ points in the same direction as $u(ty)$ for $t>0$, as long as $|y|,|ty|\in[\frac{s_1}{2},s_2]$ (by integrating $\partial_{\nu_0} u$).  Additionally, we have that
\[
\partial_{\nu_0}|u(y)|=\frac{\alpha}{|y|}|u(y)|.
\]

Now, fixing an $\omega$ with $|\omega|=1$ and integrating, we see that for any $r\in[s_1/2,s_2]$
\[
\int_{r}^{s_2}\frac{\partial_{\nu_0}|u(t\omega)|}{|u(t\omega)|}dt=\int_{r}^{s_2}\frac{\alpha}{|t|}dt
\]
and thus that
\[
\ln\frac{|u(s_2\omega)|}{|u(r\omega)|}=\alpha\ln\frac{s_2}{r}
\]
or equivalently
\[
|u(s_2\omega)|=\Big(\frac{s_2}{r}\Big)^\alpha|u(r\omega)|.
\]

Combining this with the fact that $u(y)$ and $u\big(\frac{s_2y}{|y|}\big)$ point in the same direction, for $y$ in the annulus $A_{s_1/2}^{s_2}(0)$ we have that
\[
u(y)=\Big(\frac{|y|}{s_2}\Big)^\alpha u\Big(\frac{s_2(y)}{|y|}\Big)
\]
which shows that $u$ is conically homogeneous of degree $\alpha$ about $0$ on the annulus $A_{s_1/2}^{s_2}(0)$.

We now use a bootstrapping argument to show that $u$ is in fact homogeneous on the entire ball $B_{s_2}(0)$.

On $B_{s_2}(0)$, define the function \[u_1(y):=\big(\frac{s_2}{s_1}\big)^\alpha u\big(\frac{s_1y}{s_2}\big).\]

Because $u_1$ is a rescaling of a minimizing map, $u_1$ is energy minimizing as well.  Because $u$ is conically homogeneous on the annulus $A_{s_1/2}^{s_2}(0)$, the boundary values of $u_1$ agree with the boundary values of $u$.  Hence, by the uniqueness of energy minimizers we have that $u_1\equiv u$ on $B_{s_2}(0)$.

Now, we make two observations.  First, because $u_1$ is a rescaling of $u$, if $u_1$ is homogeneous about $0$ of degree $\alpha$ on the annulus $A_r^t(0)\subset B_{s_2}(0)$, then $u$ is homogeneous about $0$ of degree $\alpha$ on the annulus $A_{(s_1/s_2)r}^{(s_1/s_2)t}(0)\subset B_{s_1}(0)$.  In particular, if $z,\lambda z$ ($\lambda>0$) are both in the annulus $A_{(s_1/s_2)r}^{(s_1/s_2)t}(0)$, we have that
\[
\lambda^\alpha\big(\frac{s_2}{s_1}\big)^\alpha u(z)=\lambda^\alpha u_1\big(\frac{s_2z}{s_1}\big)=u_1\big(\frac{s_2\lambda z}{s_1}\big)=\big(\frac{s_2}{s_1}\big)^\alpha u(\lambda z).
\]
The middle equality uses that $\frac{s_2 z}{s_1},\frac{s_2\lambda z}{s_1}\in A_r^t(0)$ and the fact that $u_1$ is homogeneous of degree $\alpha$ on this annulus.  The outer two equalities are the definition of $u_1$.  However, cancelling the factors of $\big(\frac{s_2}{s_1}\big)^\alpha$ we see that $u$ is homogeneous on the claimed annulus.

The second observation we make is that because $u_1\equiv u$ on $B_{s_2}(0)$, if $u$ is homogeneous about $0$ of degree $\alpha$ on the annulus $A_r^t(0)\subset B_{s_2}(0)$, then $u_1$ is also homogeneous about $0$ of degree $\alpha$ on the same annulus.

These two observations together in fact imply that because $u$ is homogeneous about $0$ on the annulus $A_{s_1/2}^{s_2}(0)$, then it is homogeneous on the full ball $B_{s_2}(0)$.  To see this, let 
\[
r_0=\inf\{r:u\text{ is homogeneous of degree }\alpha\text{ on }A_r^{s_2}(0)\}.
\]
We already know that $r_0\leq\frac{s_1}{2}$, by definition.

We show that $r_0=0$.  By definition, if $r_0>0$, $u$ is homogeneous on $A_{r_0}^{s_2}(0)$.  By our second observation above, we then see that $u_1$ is homogeneous on $A_{r_0}^{s_2}(0)$, and by our first observation we then see that $u$ is therefore homogeneous on $A_{(s_1/s_2)r_0}^{s_1}(0)$.  Because $r_0<s_1$, the two annuli $A_{(s_1/s_2)r_0}^{s_1}(0)$ and $A_{r_0}^{s_2}(0)$ have nonempty intersection.  Therefore, in fact, $u$ is homogeneous on their union $A_{(s_1/s_2)r_0}^{s_2}(0)$.

However, $\frac{s_1}{s_2}r_0<r_0$ because $s_1<s_2$ and $0<r_0$.  This contradicts the definition of $r_0$; therefore $r_0=0$.

Because $r_0=0$, $u$ is homogeneous on the full ball $B_{s_2}(0)$, by the definition of homogeneity.



\end{proof}

\begin{proof}[Proof of Proposition \ref{prop4.3}]
First, we rescale $u$ so that $H_\phi(0,1)=1$.  Using Proposition \ref{calculations}, we compute
\begin{align*}
W_{1/4}^4(x)&=\int_{1/4}^4\partial_rI_\phi(x,\rho)d\rho=\int_{1/4}^42(\rho H_\phi(x,\rho))^{-1}(\xi_\phi(x,\rho)-\rho I_\phi(x,\rho)E_\phi(x,\rho))d\rho\\
&=\int_{1/4}^42(\rho H_\phi(x,\rho))^{-1}(\xi_\phi(x,\rho)-2\rho I_\phi(x,\rho)E_\phi(x,\rho)+I_\phi(x,\rho)^2H_\phi(x,\rho))d\rho\\
&=\int_{1/4}^42(\rho H_\phi(x,\rho))^{-1}\int-\phi'\bigg(\frac{|y-x|}{\rho}\bigg)|y-x|^{-1}\\
&\hspace{.75in}\bigg(|\partial_{\eta_x}u|^2-2I_\phi(x,\rho)\partial_{\eta_x}u\cdot u+(I_\phi(x,\rho)d(u,0_X))^2\bigg)dy\phantom{i}d\rho\\
&=\int_{1/4}^42(\rho H_\phi(x,\rho))^{-1}\int-\phi'\bigg(\frac{|y-x|}{\rho}\bigg)|y-x|^{-1}\\
&\hspace{.75in}\underbrace{|\nabla u(y)\circ(y-x)-I_\phi(x,\rho)u(y)|^2}_{=:\theta(y,\rho)}dy\phantom{i}d\rho
\end{align*}
(where $\eta_x=y-x=|y-x|\nu_x$ in the above computations).

Now, by the definition of $\phi$, we have that $\phi'=-2{\bf 1}_{[1/2,1]}$, so we can assume that $\frac{\rho}{2}\leq|y-x|\leq\rho$, as otherwise the inner integrand vanishes.  Since $\frac{1}{4}\leq\rho\leq4$, we have that $\frac{1}{8}\leq|y-x|\leq4$.

Introducing the function
\[
\zeta(y):=|\nabla u(y)\circ(y-x)-I_\phi(x,|y-x|)u(y)|^2,
\]
we note that to establish the proposition, we must bound the integral of $\zeta$.  Using the elementary inequality $(a+b)^2\leq2(a^2+b^2)$, the above observations, and the monotonicity of the frequency, we see that
\[
\zeta(y)\leq2\theta(y,\rho)+2|I_\phi(x,\rho)-I_\phi(x,|y-x|)|^2d^2(u(y),0_X)\leq2\theta(y,\rho)+2(W_{1/8}^4(x))^2d^2(u(y),0_X).
\]

Since (\ref{3.2fbound}) bounds $W_{1/8}^4(x)$ in terms of $\Lambda$, we have that
\[
\zeta(y)\leq2\theta(y,\rho)+CW_{1/8}^4(x)d^2(u(y),0_X).
\]

We then combine this inequality with the earlier computations to conclude
\begin{align*}
W_{1/4}^4(x)\geq&\int_{1/4}^4(\rho H_\phi(x,\rho))^{-1}\int-\phi'\bigg(\frac{|y-x|}{\rho}\bigg)|y-x|^{-1}\zeta(y)dy\phantom{|}d\rho\\
&-CW_{1/8}^4(x)\int_{1/4}^42(\rho H_\phi(x,\rho))^{-1}\underbrace{\int-\phi'\bigg(\frac{|y-x|}{\rho}\bigg)|y-x|^{-1}d^2(u(y),0_X)dy}_{=H_\phi(x,\rho)}\phantom{|}d\rho\\
\geq&\int_{1/4}^4(\rho H_\phi(x,\rho))^{-1}\int-\phi'\bigg(\frac{|y-x|}{\rho}\bigg)|y-x|^{-1}\zeta(y)dy\phantom{|}d\rho-CW_{1/8}^4(x).
\end{align*}

Using (\ref{3.2fbound}) again, we have a uniform bound $I_\phi(x,\rho)\leq C$ for every $\rho\leq4$, so we use (\ref{3.2hbound}) and (\ref{calc7}), and the assumption that $H_\phi(0,1)=1$, to conclude that $H_\phi(x,\tau)$ is uniformly bounded from below for $\tau\in[1/4,4]$.  We thus have that
\[
CW_{1/8}^4(x)\geq\int\zeta(y)\underbrace{\int_{1/4}^4-\phi'\bigg(\frac{|y-x|}{\rho}\bigg)|y-x|^{-1}d\rho}_{=:M(y)}\phantom{|}dy.
\]

Since $\phi'=-2{\bf 1}_{[1/2,1]}$, we compute that
\[
M(y)=\frac{2}{|y-x|}\big[\text{min}\{4,2|y-x|\}-\text{max}\{\frac{1}{4},|y-x|\}\big]\geq2{\bf1}_{B_2(x)\setminus B_{1/4}(x)}(y)
\]
which completes the proof.
\end{proof}

\begin{proof}[Proof of Theorem \ref{thm4.2}]
We assume without loss of generality that $r=1$ and $H_\phi(0,1)=1$, and introduce the notation
\[
W(x):=W_{1/8}^4(x)=I_\phi(x,4)-I_\phi(x,1/8)
\]
for the sake of legibility.

We also write
\[
\mu_x:=-|y-x|^{-1}\phi'(|y-x|)dy
\]
and
\[
\eta_x(y):=y-x=|y-x|\nu_x(y)\hspace{16mm}v:=x_2-x_1.
\]

From (\ref{calc3}) and (\ref{calc5}) we compute that
\begin{equation}\label{freqderiv}
\partial_vI_\phi(x,1)=2H_\phi(x,1)^{-1}\bigg[\int\la\partial_{\eta_x}u,\partial_vu\ra d\mu_x-I_\phi(x,1)\int\la u,\partial_vu\ra d\mu_x\bigg].
\end{equation}

For $\ell=1,2$, we define
\[
\mathcal{E}_\ell(z):=\partial_{\eta_{x_\ell}}u-I_\phi(x_\ell,|z-x_\ell|)u(z).
\]
Note that expressions of this form are precisely the ones that are bounded in Proposition \ref{prop4.3}.

Next, note that
\begin{align*}
\partial_vu(z)&=\nabla u(z)\circ v=\nabla u(z)\circ(z-x_1)-\nabla u(z)\circ(z-x_2)=\partial_{\eta_{x_1}}u(z)-\partial_{\eta_{x_2}}u(z)\\
&=\underbrace{(I_\phi(x_1,|z-x_1|)-I_\phi(x_2,|z-x_2|))}_{=:\mathcal{E}_3(z)}u(z)+\mathcal{E}_1(z)-\mathcal{E}_2(z).
\end{align*}

We insert this into (\ref{freqderiv}) to see that
\begin{align}
\partial_vI_\phi(x,1)&=\underbrace{2H_\phi(x,1)^{-1}\int\la(\E_1-\E_2),\partial_{\eta_x}u\ra\phantom{i}d\mu_x}_{=:(A)}-\underbrace{\frac{E_\phi(x,1)}{H_\phi(x,1)^2}\int\la(\E_1-\E_2),u\ra d\mu_x}_{=:(B)}\\
&+\underbrace{2H_\phi(x,1)^{-1}\bigg[\int\la\E_3u,\partial_{\eta_x}u\ra d\mu_x-I_\phi(x,1)\int\E_3|u|^2d\mu_x\bigg]}_{=:(C)}.\nonumber
\end{align}

We then rewrite $\E_3(z)$ as
\[
\E_3(z)=\underbrace{I_\phi(x_1,1)-I_\phi(x_2,1)}_{=:\E}+\underbrace{I_\phi(x_1,|z-x_1|)-I_\phi(x_1,1)}_{=:\E_4(z)}-\underbrace{[I_\phi(x_2,|z-x_2|)-I_\phi(x_2,1)]}_{=:\E_5(z)}.
\]

Because $\mu_x$ is supported on $B_1(x)\setminus B_{1/2}(x)$, and since $x$ is on the line segment between $x_1$ and $x_2$, we have that $|x-x_\ell|\leq\frac{1}{4}$ for $\ell=1,2$ and thus $\frac{1}{4}\leq|z-x_\ell|\leq2$ for all $z\in\text{supp}(\mu_x)$.

It is then clear that for such $z\in\text{supp}(\mu_x)$, for all $x$ along the segment between $x_1$ and $x_2$,
\[
|\E_4(z)|+|\E_5(z)|\leq W(x_1)+W(x_2).
\]

Further, by (\ref{calc1}), we see that
\begin{align*}
\int\E\la u,\partial_{\eta_x}u\ra\phantom{i}d\mu_x-I_\phi(x,1)\int\E d^2(u,0_X)d\mu_x&=\E\bigg[\int|y-x|\la\partial_{\nu_x}u, u\ra\phantom{i}d\mu_x-E_\phi(x,1)\bigg]\\
&=\E\bigg[-\int\phi'(|y-x|)\la\partial_{\nu_x}u, u\ra\phantom{i}dy-E_\phi(x,1)\bigg]\\
&\stackrel{(\ref{calc1})}{=}0
\end{align*}

We thus conclude that
\begin{align*}
(C)&\leq[W(x_1)+W(x_2)]2H_\phi(x,1)^{-1}\int(|u||\nabla u|+|u|^2)d\mu_x\\
&\leq[W(x_1)+W(x_2)]2H_\phi(x,1)^{-1}\bigg(2H_\phi(x,1)+\int|\nabla u|^2d\mu_x\bigg)\\
&\leq [W(x_1)+W(x_2)](1+CH_\phi(x,1)^{-1}E_\phi(x,2))
\end{align*}
where the constant $C$ depends on the Lipschitz bound of $\phi$.  By (\ref{3.2fbound}), we have $I_\phi(x,4)\leq C(m,\phi,\Lambda)$ so that by (\ref{calc7}) \[H_\phi(x,1)^{-1}E_\phi(x,2)\leq CH_\phi(x,2)/H_\phi(x,1)\leq C.\]  Thus, $(C)\leq C[W(x_1)+W(x_2)]$, as desired.

To bound $(A)$, we see that
\begin{align*}
(A)^2&\leq4H_\phi(x,1)^{-2}\int|\E_1-\E_2|^2d\mu_x\int|\partial_{\eta_x}u|^2d\mu_x\\
&\leq4H_\phi(x,1)^{-2}\int|\E_1-\E_2|^2d\mu_x\int|\nabla u|^2d\mu_x.
\end{align*}

Using (\ref{3.2fbound}) again, we see that $I_\phi(x,\rho)\leq C$ for $\rho\leq4$, and we then use (\ref{3.2hbound}), (\ref{calc7}), and the fact that $H_\phi(0,1)=1$ by assumption, to conclude that for $\rho\in[1/4,4]$ (for example), $H_\phi(x,\rho)$ is uniformly bounded from below.  Hence
\[
(A)^2\leq C\bigg(\int(|\E_1|^2+|\E_2|^2)d\mu_x\bigg)\int|\nabla u|^2d\mu_x
\]
where $C$ depends on this lower bound on $H_\phi$.  We have also used the bound $|\E_1-\E_2|^2\leq2(|\E_1|^2+|\E_2|^2)$.

Finally, we use the same bounds used above, for $(C)$ to bound $\int|\nabla u|^2d\mu_x$, concluding that
\[
|(A)|\leq C\bigg(\int(|\E_1|^2+|\E_2|^2)d\mu_x\bigg)^{\frac{1}{2}}.
\]

Similarly, for $(B)$, we use Cauchy-Schwarz to see that
\begin{align*}
(B)^2\leq E_\phi(x,1)^2H_\phi(x,1)^{-4}\int(|\E_1-\E_2|^2)d\mu_x\int|u|^2d\mu_x.
\end{align*}
The second integral here is simply $H_\phi(x,1)$, and we have already seen how to bound $\int(|\E_1-\E_2|^2)d\mu_x$ from above and $H_\phi(x,1)$ from below.  Additionally, using (\ref{calc7}), (\ref{3.2hbound}), and (\ref{3.2fbound}), we have that $E_\phi(x,1)\leq H_\phi(x,1)I_\phi(x,1)\leq C$.  So we also have that
\[
|(B)|\leq C\bigg(\int(|\E_1|^2+|\E_2|^2)d\mu_x\bigg)^{\frac{1}{2}}.
\]

Combining the bounds on $(A)$, $(B)$, and $(C)$, we have
\[
\partial_vI_\phi(x,1)\leq C(W(x_1)+W(x_2))+C\bigg(\int(|\E_1|^2+|\E_2|^2)d\mu_x\bigg)^{\frac{1}{2}}.
\]

To conclude, it suffices to prove the bound
\begin{equation}\label{4.2goal}
\int|\E_\ell|^2d\mu_x\leq CW(x_\ell).
\end{equation}

For $y\in\text{supp}(\mu_x)$, we have that $\frac{1}{2}\leq|y-x|\leq1$, and since (as earlier) $|x-x_\ell|\leq\frac{1}{4}$, we have that $\frac{1}{4}\leq|y-x_\ell|\leq\frac{5}{4}$.  Also, recall that $d\mu_x(y)=-\phi'(|y-x|)|y-x|^{-1}dy$ and that for $y\in\text{supp}(\mu_x)$, $0\leq-\phi'(|y-x|)|y-x|^{-1}\leq4$ by our choice of $\phi$.  Hence, $-\phi'(|y-x|)|y-x|^{-1}\leq4{\bf 1}_{B_2(x_\ell)\setminus B_{1/4}(x_\ell)}(y)$.

Therefore,
\[
\int|\E_\ell|^2d\mu_x\leq4\int_{B_2(x_\ell)\setminus B_{1/4}(x_\ell)}|\E_\ell|^2dy
\]
and (\ref{4.2goal}) then follows from Proposition \ref{prop4.3}.

We thus obtain the bound
\[
\partial_vI_\phi(x,1)\leq C(W(x_1)+W(x_2))+C\big[W(x_1)+W(x_2)\big]^{\frac{1}{2}}
\]
and since we know that the frequencies here are uniformly bounded in terms of $\Lambda$, the second term bounds the first.  Thus,
\[
\partial_vI_\phi(x,1)\leq C\big[W(x_1)+W(x_2)\big]^{\frac{1}{2}}
\]
for each $x$ on the line segment between $x_1$ and $x_2$.  Reversing $x_1$ and $x_2$, we find that
\[
|\partial_vI_\phi(x,1)|\leq C\big[W(x_1)+W(x_2)\big]^{\frac{1}{2}}.
\]
The conclusion of the theorem follows by integrating this inequality between any points on the line segment between $x_1$ and $x_2$.
\end{proof}

\section{Mean Flatness}\label{Jones5}
In this section, we prove that the frequency pinching bounds the mean flatness of certain measures, on balls that intersect $\A$ nontrivially.

\begin{definition}\label{def:Jones}
For a Radon measure $\mu$ on $\R^m$ and $k\in\{0,1,\dots,m-1\}$, $x\in\R^m$ and $r>0$, we define the {\bf $k^{\text{th}}$ mean flatness} of $\mu$ in the ball $B_r(x)$ to be
\[
D_\mu^k(x,r):=\inf_Lr^{-k-2}\int_{B_r(x)}\text{dist}(y,L)^2\phantom{i}d\mu(y),
\]
where the infimum is taken over all affine $k$-planes $L$.
\end{definition}

Another characterization of the mean flatness is in terms of a bilinear form.  Suppose that $B_{r_0}(x_0)$ is such that $\mu(B_{r_0}(x_0))>0$, and $\overline{x}_{x_0,r_0}$ is the {\bf barycenter} of $\mu$ in this ball,
\[
\overline{x}_{x_0,r_0}:=\frac{1}{\mu(B_{r_0}(x_0))}\int_{B_{r_0}(x_0)}x\phantom{i}d\mu(x).
\]

Consider the positive semi-definite form
\[
b(v,w):=\int_{B_{r_0}(x_0)}((x-\overline{x}_{x_0,r_0})\cdot v)((x-\overline{x}_{x_0,r_0})\cdot w)\phantom{i}d\mu(x).
\]

We can diagonalize this bilinear form by an orthonormal basis of $\R^m$; that is, we can find $\{v_1,\dots,v_m\}$ so that the $v_i$ are an orthonormal basis of $\R^m$ ($v_i\cdot v_j=\delta_{ij}$), and so that
\[
b(v_i,v_i)=\lambda_i\hspace{5mm}\text{and}\hspace{5mm}b(v_i,v_j)=0\text{ if }i\neq j
\]
for some $0\leq\lambda_m\leq\lambda_{m-1}\leq\dots\leq\lambda_1$.

We then have that
\begin{equation}\label{eigeneq}
\int_{B_{r_0}(x_0)}((x-\overline{x}_{x_0,r_0})\cdot v_i)x\phantom{i}d\mu(x)=\lambda_iv_i
\end{equation}
for each $1\leq i\leq m$.

The $k^{\text{th}}$ mean flatness of $\mu$ in the ball $B_{r_0}(x_0)$ can then be computed as
\[
D_\mu^k(x_0,r_0)=r_0^{-k-2}\sum_{\ell=k+1}^m\lambda_\ell
\]
and the infimum in the definition of the mean flatness is achieved by any affine $k$-plane $L=\overline{x}_{x_0,r_0}+\text{Span}\{v_1,\dots,v_k\}$, where $v_1,\dots,v_m$ form an orthonormal eigenbasis with eigenvalues $\lambda_1\geq\lambda_2\geq\dots\geq\lambda_m$.

We have the following bound on the $(m-2)^{\text{th}}$ mean flatness of a measure $\mu$ by the frequency pinching, on balls that intersect $\A$, the set of points of high order.

\begin{proposition}\label{flatbound}
Under Assumptions \ref{cone} and \ref{bound}, there exists a positive constant $C(\Lambda,m,X)$ so that, for a finite nonnegative Radon measure $\mu$, and a ball $B_{r/8}(x_0)$ such that $B_{r/8}(x_0)\cap\A$ is nonempty, then
\[
D_\mu^{m-2}(x_0,r/8)\leq\frac{C}{r^{m-2}}\int_{B_{r/8}(x_0)}W_{r/8}^{4r}(x)\phantom{i}d\mu(x).
\]
\end{proposition}

We remark that this result is stated in a slightly more general manner than the corresponding result of \cite{DMSV}, Proposition 5.3.  In particular, rather than assuming that $\text{spt}(\mu)$ is a subset of $\A$, we assume only that $B_{r/8}(x_0)\cap\A$ is nonempty.  However, a careful inspection of the proof in \cite{DMSV} reveals that the latter assumption is all that is needed in their proof.

To show this, we need the following lemma, which we can think of as a (very weak) unique extension lemma for certain harmonic maps into $F$-connected complexes.

\begin{lemma}\label{extension}
If $x_1,\dots,x_m$ are orthonormal coordinates for $\R^m$, and a minimizing map $u:\Omega\to X$ is a function of $x_1$ alone on $B_r(x_0)\subset\Omega$, $\Omega\subset\R^m$ is convex, and $X$ is an $F$-connected complex, then $u$ is a function of $x_1$ alone on all of $\Omega$.
\end{lemma}

We remark that this result need not hold for non-convex domains.  Consider the domain shown in Figure \ref{fig:nonunique}, $\Omega=I\cup II$ for the indicated regions $I$ and $II$.

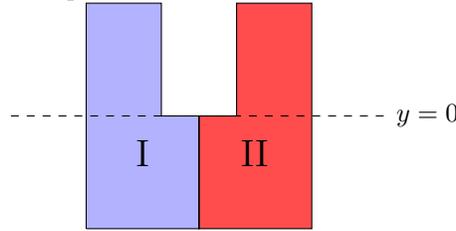
\begin{figure}\centering
Non-unique extension for a non-convex domain

\begin{tikzpicture}
\fill[blue!30!white](0,0) rectangle (1,3);
\fill[blue!30!white](1,0) rectangle (1.5,1.5);
\fill[red!70!white](1.5,0) rectangle (2,1.5);
\fill[red!70!white](2,0) rectangle (3,3);
\draw (.75,1) node {\LARGE I};
\draw (2.24,1) node {\LARGE II};
\draw[black](0,0)--(0,3)--(1,3)--(1,1.5)--(2,1.5)--(2,3)--(3,3)--(3,0)--(1.5,0)--(1.5,1.5)--(1.5,0)--(0,0);
\draw[dashed](-1,1.5)--(4,1.5) node[anchor=west] {$y=0$};
\end{tikzpicture}
\caption{The region shown is a square, with a rectangle removed from the top-middle, so that the top half of the square is broken into two pieces.  The left of the region is labelled ``I" and the right is labelled ``II," so that these meet in the middle of the lower half of the square and are separated by the removed rectangle in the upper half.  The line dividing the lower and upper halves of this square is marked as the $x$-axis, $y=0$.}\label{fig:nonunique}
\end{figure}

We shall define a harmonic map from this region $\Omega$ of the plane into the tripod $Y$, which consists of three rays joined at a common vertex $0_Y$.  Call these rays $R_0,R_1,R_2$.  Let $\gamma_1(x)$ be the unit speed geodesic that passes through $R_0$ and $R_1$, with $\gamma_1(0)$ defined to be the vertex point $0_Y$, and $\gamma_2(x)$ be the unit speed geodesic passing through $R_0$ and $R_2$, with $\gamma_2(0)$ also being $0_Y$.  We consider both of these geodesics to map the negative reals into $R_0$ and the positive reals into $R_1,R_2$ (respectively).

Now, consider the map $u:\Omega\to Y$ defined by setting $u(x,y)=\gamma_1(y)$ for $(x,y)\in I$, the left part of $\Omega$, and $u(x,y)=\gamma_2(y)$ for $(x,y)\in II$.  We then note that $u$ is a minimizing map from $\Omega$ into an $F$-connected complex, and $u$ is not a function of $y$ alone (because its values in the upper half of $\Omega$ very much depend on $x$).  However, $u$ is a function of $y$ alone on any subset of the lower half of $\Omega$.

This is not terribly surprising, because harmonic functions into $F$-connected complexes do not usually have unique extension properties (not even geodesics, as we see here, are uniquely extensible).  Fortunately, for our purposes it suffices to work only with convex domains.

We shall also need the following useful lemma, which, speaking roughly, is a compactness principle for maps into $F$-connected complexes.  Crucially, this lemma also ensures that if a sequence of maps $u_k$ all have points of high order $x_k$, the limit map also has a point of high order.

\begin{lemma}\label{compactness}
Suppose that $u_k:B_{64}(0)\to X$ are minimizing maps into an $F$-connected complex $X$ with $u_k(0)=0_X$ and so that \begin{enumerate} \item $I_{\phi,u_k}(0,64)$ and $H_{\phi,u_k}(0,64)$ are uniformly bounded (in $k$), and \item $H_{\phi,u_k}(0,64)$ is uniformly bounded away from $0$.
\end{enumerate}

Then there is a subsequence $u_\ell$ of the $u_k$ converging uniformly in $B_{16}(0)$ and strongly in $H^1(B_{16}(0))$ to a nonconstant minimizing map $u:B_{16}(0)\to X$.

Moreover, if we suppose that for each $u_k$, there exists some $x_k\in B_{16}(0)$ with $\text{Ord}_{u_k}(x_k)>1$, and $x_k\to x$, then $\text{Ord}_u(x)>1$ as well.  
\end{lemma}

We remark that while Lemma \ref{compactness} concludes that the limit map $u$ has a point of high order, it does not claim that this point is singular, even if all of the $x_k$ are singular points of high order.  This weaker claim will be enough for our purposes due to part (4) of Theorem \ref{thm:BigGS}.  

\begin{proof}[Proof of Lemma \ref{extension}]
We begin by noting that on $B_r(x_0)$, $u(x)=\gamma(x_1)$ for some constant speed geodesic $\gamma$.  In particular, $\mathcal{S}(u)$, the singular set, does not intersect $B_r(x_0)$, and by part (3) of Theorem \ref{thm:BigGS} we have that $\mathcal{S}(u)$ is a relatively closed set of codimension $2$.  

For any regular point $y$, there is a path from $x_0$ to $y$ that does not intersect $\Si(u)$.  We cover this path by balls $B_i$ such that $u$ is given by an analytic function $\R^m\to\R^n$ on each of these balls (where $\R^n$ should be understood to denote one of the totally geodesic $n$-flats of $X$, which are not necessarily the same for all of the balls \textit{a priori}).  By compactness, we can cover this path with a finite number $B_r(x)=B_1,\dots,B_k$ of balls where we can assume that $B_i\cap B_{i+1}$ is nonempty and $y\in B_k$.  Inductively assuming that $u|_{B_i}$ is a function of $x_1$ , we observe that by unique extension, $u|_{B_{i+1}}$ is a function of $x_1$ alone as well.  Hence, we have that $u$ is {\em locally} a function of $x_1$ alone away from $\Si(u)$.

We wish to show that $u$ is constant on each affine $(m-1)$-plane of the form $P_x=x+\text{span}\{x_2,\dots,x_m\}\cap\Omega$, for $x$ on the line spanned by $x_1$.  By the convexity of $\Omega$, these planes $P_x$ are all connected.  Additionally, $u$ is continuous by Theorem \ref{GS2.3}.

For almost every $x$, $P_x\cap\Si(u)$ is of Hausdorff dimension less than $m-2$, for otherwise the full singular set would have Hausdorff dimension at least $m-1$.  (This is a simple consequence of, e.g. \cite{Fe}, result 2.10.25.)  
In particular, for almost every $x$, the set $P_x\cap\Si(u)$ is of topological (or inductive) dimension at most $m-3$, and therefore does not disconnect $P_x$.  

If $P_x\cap\Si(u)$ does not disconnect $P_x$, write $x\in G$.  The preceding paragraph demonstrates that $G$ has null complement.

For such $x\in G$, since $u$ is locally constant on an open, dense, connected subset of $P_x$, $u|_{P_x}$ is constant.

For $x\notin G$, we note that $x$ is a limit point of $G$, because $x$ lies in a set of measure zero.  Let $x_n$ be a sequence in $G$ converging to $x$, and note that $u|_{P_{x_n}}$ is constant.  Moreover, because $u$ is continuous, $u|_{P_x}$ is the limit of the $u|_{P_{x_n}}$ and hence is also a constant function.

Thus, $u$ is a function of $x_1$ alone on all of $\Omega$.

\end{proof}

\begin{proof}[Proof of Lemma \ref{compactness}]
First, we note that since $I_{\phi,u_k}(0,64)$ and $H_{\phi,u_k}(0,64)$ are both uniformly bounded, the smoothed energies $E_{\phi,u_k}(0,64)$ are likewise uniformly bounded.  We then immediately see that the unsmoothed energies $\int_{B_{32}(0)}|\nabla u_k|^2$ are also uniformly bounded, although we must shrink the size of the ball to obtain this.  By Theorem \ref{GS2.3}, we then obtain a uniform Lipschitz bound on the $u_k$ on $B_{16}(0)$, shrinking the ball again so that we are in the interior of the region where we have a uniform bound on the total energy..  We then apply Arzela-Ascoli to conclude that a subsequence $u_k$ (unrelabeled) converges uniformly to some $u:B_{16}(0)\to X$.  It is clear that $u(0)=0_X$ because $u_k(0)=0_X$ for all $k$.

By uniform convergence, we see that \[\int_{B_{16}(0)}|u|^2=\lim_{k\to\infty}\int_{B_{16}(0)}|u_k|^2.\] Moreover, the lower semicontinuity of energy tells us that
\[
E(u):=\int_{B_{16}(0)}|\nabla u|^2\leq\liminf_{k\to\infty}\int_{B_{16}(0)}|\nabla u_k|^2=:\liminf_{k\to\infty}E(u_k)
\]
where we can of course choose a minimizing subsequence to replace the $\liminf$ above with a limit.  We seek to show that equality holds above, that is, that there is no ``energy drop."

We shall do this by contradiction.  Assuming that the energy of $u$ is strictly less than the limit of the $u_k$, we directly find competitors to $u_k$ which (for sufficiently large $k$) have less energy than $u_k$, contradicting the minimality of $u_k$.  In particular, suppose that \begin{equation}\label{eq:EneDro}E(u)+\epsilon<\liminf_{k\to\infty}E(u_k)\end{equation} for some $\epsilon>0$.

For $\delta>0$ (which we shall treat as undetermined for now and fix later), we define a new map $u_{\delta,k}:B_{16}(0)\to X$.

First, for $\omega\in\partial B_{16}(0)$, we define the geodesic $\gamma_{k,\omega}:[16-\delta,16]\to X$ which connects $u(\omega)$ to $u_k(\omega)$ (Observe that the domain of $\gamma_{k,\omega}$ is unusual).  We then define
\[
u_{\delta,k}(x)=\begin{cases}u\Big(\frac{16x}{16-\delta}\Big)&\text{on }B_{16-\delta}(0)\\
\gamma_{k,(16x/|x|)}(|x|)&\text{on }B_{16}(0)\setminus B_{16-\delta}(0).

\end{cases}
\]

On $\partial B_{16-\delta}(0)$, $u_{\delta,k}$ has values given by the boundary values of $u$ on $\partial B_{16}(0)$, and the map is simply a rescaling of $u$ to this smaller ball.  The energy of $u_{\delta,k}|_{B_{16-\delta}(0)}$ is hence $\big(\frac{16-\delta}{16}\big)^{m-2}E(u)$.

By definition, on the radial segment from $\frac{16-\delta}{16}\omega$ to $\omega\in\partial B_{16}(0)$, $u_{\delta,k}$ traces out the geodesic from $u(\omega)$ to $u_k(\omega)$, at constant speed.

We note, moreover, that if $x,y\in B_{16}(0)$, $r\in\big[\frac{16-\delta}{16},1\big]$, we have (by quadrilateral comparisons as in Korevaar-Schoen \cite{KSSobolev}, which come from repeatedly applying the nonpositive curvature condition)
that
\[
|u_{\delta,k}(rx)-u_{\delta,k}(ry)|\leq|u(x)-u(y)|+|u_k(x)-u_k(y)|.
\]

Fixing $r$, dividing by $|rx-ry|$, and taking the limit as $x\to y$ along the great circle between them, we see that
\[
|\partial_vu_{\delta,k}(y)|\leq r^{-1}|\partial_{v}u(y)|+r^{-1}|\partial_{v}u_k(y)|
\]
where $v$ denotes the unit vector tangent to the great circle connecting $x$ and $y$.  Note that by changing our choice of $x$, we can take this $v$ to be any vector tangential to $B_r(0)$ at $y$.  The Lipschitz bounds on $u_k$ and $u$, together with the uniform lower bounds on $r$ (as long as we take $\delta<16$), imply that $|\partial_vu_{\delta,k}(y)|$ is uniformly bounded.  In particular, the tangential energy densities of $u_{\delta,k}$ are bounded pointwise on the annulus $B_{16}(0)\setminus B_{16-\delta}(0)$.  By taking $\delta$ sufficiently small, we can hence take the total tangential energy in this annulus to be less than $\frac{\epsilon}{2}$ for all $k$.


If $m$ is exactly $1$, we also, at this stage, choose $\delta$ to be small enough that $\big(\frac{16-\delta}{16}\big)^{-1}E(u)$ is at most $E(u)+\frac{\epsilon}{4}$.  For other spatial dimensions, the rescaling of $u$ does not increase the energy, as in these cases $m-2\geq0$.  In any event, we now fix $\delta$ subject to these constraints and hold it constant for the rest of the proof.

We now observe immediately that, on $B_{16}(0)\setminus B_{16-\delta}(0)$,
\[
\Big|\frac{\partial}{\partial r}u_{\delta,k}\Big|\leq\delta^{-1}\|u-u_k\|_{L^\infty(\partial B_{16}(0))}.
\]Since we have fixed $\delta$, this is some constant times $\|u-u_k\|_{L^\infty(\partial B_{16}(0))}$.  As we take $k\to\infty$, since the $u_k$ converge uniformly to $u$ on $B_{16}(0)$, the radial energy of $u_{\delta,k}$ on $B_{16}(0)\setminus B_{16-\delta}(0)$ approaches $0$.

In particular, for sufficiently large $k$, the energy of $u_{\delta,k}$ is strictly less than the energy of $u_k$, contradicting the fact that $u_k$ is an energy minimizer.

In fact, we can use this same construction to argue that $u$ itself is an energy minimizer, for if it is not, by replacing $u$ in the above construction with the energy minimizer for the boundary values of $u$, we reach the same contradiction by the same method.

To see that $u$ is nonconstant, we note that by (\ref{calc7}) we have that
\[
16^{1-m}H_{\phi,u_k}(0,16)=64^{1-m}H_{\phi,u_k}(0,64)\exp\bigg(-2\int_{16}^{64}I_{\phi,u_k}(x,t)\frac{dt}{t}\bigg).
\]
so that the uniform bounds from above on $I_{\phi,u_k}(0,64)$, along with the fact that $H_{\phi,u_k}(0,64)$ is uniformly bounded from below away from $0$, imply that $H_{\phi,u_k}(0,16)$ is uniformly bounded from below, away from $0$.  Hence, $H_{\phi,u}(0,16)>0$, so that since $u$ is continuous and $u(0)=0_X$, we see that $u$ must be nonconstant.  Proposition \ref{prop:GS3.4} then allows us to conclude that $u$ is nonconstant on any nonempty open subset of $B_{16}(0)$.

We now wish to show that this convergence is strong in $H^1(B_{16}(0))$.  We recall that, following Gromov-Schoen \cite{GS}, we have identified our complex $X$ as a closed subset of $\R^N$ for some $N$.  

We have uniform $H^1(B_{16}(0),\R^N)$ bounds on the $u_k$, since the $L^2$ bounds are implied by the bounds on $H_\phi$ and the monotonicity of $H_\phi$; the $L^2$ bounds on the derivative come from the energy.  Hence, we can take the $\{u_k\}$ to converge weakly to some $v\in H^1$ up to taking a subsequence.  Since $L^2$ convergence implies pointwise almost everywhere convergence up to a subsequence, we see that $v$ and $u$ must in fact agree.  Hence, the $u_k$ converge weakly to $u$ in $H^1(B_{16}(0))$, up to a subsequence.

However, we also have that \[\lim_{k\to\infty}\|u_k\|_{H^1}=\|u\|_{H^1}.\]  Here, the zeroth-order part of this norm converges by the uniform convergence of the $u_k$ to $u$, and the first-order convergence is simply the aforementioned convergence of the energies (i.e. the fact that we do not have an ``energy drop").  Then, since we have weak convergence which respects the norm of $H^1$, we see that this subsequence of the $u_k$ converges strongly to $u$ in $H^1(B_{16}(0),\R^N)$.

Hence, the convergence of the $u_k$ to $u$ is both uniform and strong in $H^1(B_{16}(0))$.  

We now wish to see that, if all of the $u_k$ have a point of high order, then the limit map $u$ has a point of high order as well.  In particular, we consider the unsmoothed order (cf. Definition \ref{def:order})
\[
\text{Ord}_v(x,r,P):=\frac{rE_v(x,r)}{H_v(x,r,P)}.
\]

If $x_k\to x$ and $P_k\to P$, we have that
\begin{equation}\label{OrdCont}
\lim_{k\to\infty}\text{Ord}_{u_k}(x_k,r,P_k)=\text{Ord}_u(x,r,P)
\end{equation}
due to the uniform and strong $H^1$ convergence, using standard arguments.  We recall (\ref{def:ctsord})\begin{equation*}\text{Ord}_u(x)=\lim_{r\to0}\text{Ord}_u(x,r,u(x)).\end{equation*}

The uniform Lipschitz bounds on $u,u_k$ imply that both $u(B_{16}(0))$ and $u_k(B_{16}(0))$ lie in some bounded (and hence compact) $K\subset X$.  Hence, by part (1) of Theorem \ref{thm:BigGS}, there is some $\epsilon_X>0$ so that for all $x$ in $B_{16}(0)$, either $\text{Ord}_v(x)=1$ or $\text{Ord}_v(x)\geq1+\epsilon_X$ where $v$ may denote any minimizing map whose image lies in $K$.  Since we have assumed that each $u_k$ has at least one point of order strictly greater than $1$, we select for each $u_k$ an $x_k$ with $\text{Ord}_{u_k}(x_k)\geq1+\epsilon_X$.  We assume that these $x_k$ converge to some $x\in B_{16}(0)$, by passing to a subsequence (unrelabeled).

Thus, using (\ref{def:ctsord}) and the fact that the order is nondecreasing as $r$ increases, we have that for each $r>0$ and for each $k$
\[
1+\epsilon_X\leq\text{Ord}_{u_k}(x_k)\leq\text{Ord}_{u_k}(x_k,r,u_k(x_k)).
\]
We therefore have, taking a limit in $k$ and using (\ref{OrdCont}), that for each $r>0$,
\[
1+\epsilon_X\leq\text{Ord}_u(x,r,u(x))
\]
because the $u_k$ converge uniformly to $u$ and hence $u_k(x_k)$ converges to $u(x)$.  Taking the limit in $r$ then shows that $x$ is a point of order strictly greater than $1$ for $u$, which is what we wished to find.
\end{proof}

\begin{proof}[Proof of Proposition \ref{flatbound}]
The scale-invariance of the mean flatness allows us to assume that $r=1$ and $H_\phi(0,1)=1$.  Let $\overline{x}$ be the barycenter of $\mu$ in $B_{1/8}(x_0)$, and let $\{v_1,\dots,v_m\}$ diagonalize the bilinear form $b$ associated to $\mu$, with corresponding eigenvalues $0\leq\lambda_m\leq\lambda_{m-1}\leq\dots\leq\lambda_1$.  Then, the definition of the barycenter and (\ref{eigeneq}) let us conclude that, for each $v_j$ and every $z\in\Omega$,
\[\label{flatboundeq1}
 \nabla u(z)\circ(-\lambda_jv_j)=\int_{B_{1/8}(x_0)}((x-\overline{x})\cdot v_j)( \nabla u(z)\circ(z-x)-\alpha u(z))\phantom{i}d\mu(x)
\]
for any constant $\alpha$; we shall choose $\alpha$ later in the proof.  (It is helpful to compare the right-hand side of this equation to Proposition \ref{prop4.3}, for motivation.)

Squaring this equation and using, again, (\ref{eigeneq}) and the definition of the barycenter, we see that
\begin{align*}
\lambda_j^2|\partial_{v_j}u|^2&\leq\bigg(\int_{B_{1/8}(x_0)}|(x-\overline{x})\cdot v_j|\phantom{i}| \nabla u(z)\circ(z-x)-\alpha u(z)|\phantom{i}d\mu(x)\bigg)^2\\
&\leq\int_{B_{1/8}(x_0)}((x-\overline{x})\cdot v_j)^2\phantom{i}d\mu(x)\int_{B_{1/8}(x_0)}|\nabla u(z)\circ(z-x)-\alpha u(z)|^2\phantom{i}d\mu(x)\\
&=\lambda_j\int_{B_{1/8}(x_0)}|\nabla u(z)\circ(z-x)-\alpha u(z)|^2\phantom{i}d\mu(x).
\end{align*}

Thus,
\[
\lambda_j|\partial_{v_j}u|^2\leq\int_{B_{1/8}(x_0)}|\nabla u(z)\circ(z-x)-\alpha u(z)|^2\phantom{i}d\mu(x)
\]
where $z\in\Omega$ and $\alpha\in\R$ are (at present) arbitrary.

Recalling the earlier characterization of the $(m-2)^{\text{th}}$ mean flatness, we sum from $j=1$ to $m-1$, and integrate over $z\in B_{5/4}(x_0)\setminus B_{3/4}(x_0)$ to conclude that
\begin{align*}
D_\mu^{m-2}(x_0,1/8)&\int_{B_{5/4}(x_0)\setminus B_{3/4}(x_0)}\sum_{j=1}^{m-1}|\partial_{v_j}u(z)|^2\phantom{i}dz\\
&=\bigg(\frac{1}{8}\bigg)^{-m }\int_{B_{5/4}(x_0)\setminus B_{3/4}(x_0)}(\lambda_{m-1}+\lambda_m)\sum_{j=1}^{m-1}|\partial_{v_j}u(z)|^2\phantom{i}dz\\&\leq C\int_{B_{5/4}(x_0)\setminus B_{3/4}(x_0)}\lambda_{m-1}\sum_{j=1}^{m-1}|\partial_{v_j}u(z)|^2\phantom{i}dz\\
&\leq C\int_{B_{5/4}(x_0)\setminus B_{3/4}(x_0)}\sum_{j=1}^{m-1}\lambda_j|\partial_{v_j}u(z)|^2\phantom{i}dz\\
&\leq C\int_{B_{5/4}(x_0)\setminus B_{3/4}(x_0)}\int_{B_{1/8}(x_0)}|\nabla u(z)\circ(z-x)-\alpha u(z)|^2\phantom{i}d\mu(x)\phantom{i}dz\\
&\leq C\int_{B_{1/8}(x_0)}\int_{B_{3/2}(x)\setminus B_{1/2}(x)}|\nabla u(z)\circ(z-x)-\alpha u(z)|^2\phantom{i}dz\phantom{i}d\mu(x).\\
\end{align*}

Now, we claim that
\begin{equation}\label{flatboundclaim}
\int_{B_{5/4}(x_0)\setminus B_{3/4}(x_0)}\sum_{j=1}^{m-1}|\partial_{v_j}u(z)|^2\phantom{i}dz\geq c(\Lambda)>0
\end{equation}
so that
\[
D_\mu^{m-2}(x_0,1/8)\leq C\int_{B_{1/8}(x_0)}\int_{B_{3/2}(x)\setminus B_{1/2}(x)}|\nabla u(z)\circ(z-x)-\alpha u(z)|^2\phantom{i}dz\phantom{i}d\mu(x).
\]

Suppose that (\ref{flatboundclaim}) failed.  We could then find a sequence of minimizing maps $u_k$ satisfying Assumptions \ref{cone} and \ref{bound}, so that each $u_k$ had a point of high order $x\in B_{1/8}(x_0)\cap\A$ and so that
\[
\int_{B_{5/4}(x_0)\setminus B_{3/4}(x_0)}\sum_{j=1}^{m-1}|\partial_{v_j}u_k(z)|^2\phantom{i}dz\leq\frac{1}{k}.
\]
We take these maps to have $H_{\phi,u_k}(0,1)=1$, which implies a uniform upper bound on $H_{\phi,u_k}(0,64)$ using (\ref{calc7}), and also implies that $H_{\phi,u_k}(0,64)\geq1$.

Using Lemma \ref{compactness} we see that, up to a subsequence, these $u_k$ converge to a nonconstant minimizing map $u:B_{16}(0)\to X$ so that
\[
\int_{B_{5/4}(x_0)\setminus B_{3/4}(x_0)}\sum_{j=1}^{m-1}|\partial_{v_j}u(z)|^2\phantom{i}dz=0.
\]
We note that this $u$ is therefore a function of one variable only on any ball contained within the annular region $B_{5/4}(x_0)\setminus B_{3/4}(x_0)$.  Applying Lemma \ref{extension}, $u$ must in fact be a function of one variable alone on its whole domain.

However, again by Lemma \ref{compactness}, $u$ has a point $x$ having $\text{Ord}_u(x)>1$, and since $u$ depends only on one variable, this means that there is in fact an $(m-1)$ dimensional affine subspace of such points, contradicting part (4) of Theorem \ref{thm:BigGS}.

Hence, (\ref{flatboundclaim}) holds.  We note that while this inequality is given in terms of $x_0$ and the direction vectors $v_1,\dots,v_j$, the constant $c(\Lambda)$ does not depend on these.  To see this, simply apply an isometry of $\R^n$ translating $x_0$ to any other $y_0$ and taking the orthonormal vectors $v_j$ to any desired set of orthonormal vectors.

Thus, we have that
\[
D_\mu^{m-2}(x_0,1/8)\leq C\int_{B_{1/8}(x_0)}\int_{B_{3/2}(x)\setminus B_{1/2}(x)}|\nabla u(z)\circ(z-x)-\alpha u(z)|^2\phantom{i}dz\phantom{i}d\mu(x),
\]
and using the triangle inequality we conclude that
\begin{align}\label{5.4goal}
\notag D_\mu^{m-2}(x_0,1/8)\leq C&\underbrace{\int_{B_{1/8}(x_0)}\int_{B_{3/2}(x)\setminus B_{1/2}(x)}|\nabla u(z)\circ(z-x)-I_\phi(x,1)u(z)|^2\phantom{i}dz\phantom{i}d\mu(x)}_{=:(I)}\\
&+C\underbrace{\int_{B_{1/8}(x_0)}\int_{B_{3/2}(x)\setminus B_{1/2}(x)}\big(I_\phi(x,1)-\alpha\big)^2|u(z)|^2\phantom{i}dz\phantom{i}d\mu(x)}_{=:(II)}.
\end{align}

We now choose $\alpha$ with the intent of controlling $(II)$.  In particular, we set
\[
\alpha=\frac{1}{\mu(B_{1/8}(x_0))}\int_{B_{1/8}(x_0)}I_\phi(y,1)\phantom{i}d\mu(y).
\]
This allows us to bound $(II)$ by first evaluating the inner integral in $z$,
\[
\int_{B_{3/2}(x)\setminus B_{1/2}(x)}|u(z)|^2dz,
\]
which we can bound by $H_\phi(x,1)+H_\phi(x,2)$, and this we can bound in terms of $H_\phi(0,1)$ using (\ref{3.2hbound}) and (\ref{calc7}).
\begin{align*}
(II)&\leq CH_\phi(0,1)\int_{B_{1/8}(x_0)}\bigg(I_\phi(x,1)-\frac{1}{\mu(B_{1/8}(x_0))}\int_{B_{1/8}(x_0)}I_\phi(y,1)\phantom{i}d\mu(y)\bigg)^2\phantom{i}d\mu(x)\\
&=C\int_{B_{1/8}(x_0)}\bigg(\frac{1}{\mu(B_{1/8}(x_0))}\int_{B_{1/8}(x_0)}(I_\phi(x,1)-I_\phi(y,1))d\mu(y)\bigg)^2d\mu(x)\\
&\leq\frac{C}{\mu(B_{1/8}(x_0))}\int_{B_{1/8}(x_0)}\int_{B_{1/8}(x_0)}(I_\phi(x,1)-I_\phi(y,1))^2\phantom{i}d\mu(y)\phantom{i}d\mu(x).
\end{align*}
Hence, applying Theorem \ref{thm4.2} we see that
\begin{align}\label{5.4b1}
\notag(II)&\leq\frac{C}{\mu(B_{1/8}(x_0))}\int_{B_{1/8}(x_0)}\int_{B_{1/8}(x_0)}\big(W_{1/8}^4(x)+W_{1/8}^4(y)\big)\phantom{i}d\mu(y)\phantom{i}d\mu(x)\\
&=2C\int_{B_{1/8}(x_0)}W_{1/8}^4(x)\phantom{i}d\mu(x).
\end{align}

We again use the triangle inequality to bound $(I)$,
\begin{align*}
(I)\leq &C\underbrace{\int_{B_{1/8}(x_0)}\int_{B_{3/2}(x)\setminus B_{1/2}(x)}\big|I_\phi(x,1)-I_\phi(x,|z-x|)\big|^2|u(z)|^2\phantom{i}dz\phantom{i}d\mu(x)}_{=:(I_1)}\\
&+C\underbrace{\int_{B_{1/8}(x_0)}\int_{B_{3/2}(x)\setminus B_{1/2}(x)}|\nabla u(z)\circ(z-x)-I_\phi(x,|z-x|)u(z)|^2\phantom{i}dz\phantom{i}d\mu(x)}_{=:(I_2)}.
\end{align*}
For the first integral appearing here, we note that by construction, for each pair of $x,z$ in the domain of integration, $\frac{1}{2}\leq|z-x|\leq\frac{3}{2}$ and hence for each such pair,
\[
|I_\phi(x,|z-x|)-I_\phi(x,1)|\leq I_\phi(x,3/2)-I_\phi(x,1/2)\leq W_{1/8}^4(x).
\]
In particular, this means that (bounding the integral in $z$ as in the previous work we used to bound $(II)$),
\begin{equation}\label{5.4b2}
(I_1)\leq CH_\phi(0,1)\int_{B_{1/8}(x_0)}W_{1/8}^4(x)^2\phantom{i}d\mu(x)\leq C\int_{B_{1/8}(x_0)}W_{1/8}^4(x)\phantom{i}d\mu(x).
\end{equation}

As for $(I_2)$, we apply Proposition \ref{prop4.3} directly, concluding that
\[
\int_{B_{3/2}(x)\setminus B_{1/2}(x)}|\nabla u(z)\circ(z-x)-I_\phi(x,|z-x|)u(z)|^2\phantom{i}dz\leq CW_{1/8}^4(x).
\]
Thus,
\begin{equation}\label{5.4b3}
(I_2)\leq C\int_{B_{1/8}(x_0)}W_{1/8}^4(x)\phantom{i}d\mu(x),
\end{equation}

Combining (\ref{5.4goal}) with (\ref{5.4b1}), (\ref{5.4b2}), and (\ref{5.4b3}) we conclude the desired result.
\end{proof}

\section{Qualitative Comparison to Homogeneous Maps}\label{Qualitative6}
We next show that if a minimizing function $u$ has sufficiently small pinching at points which are spread out enough, then $\A$, and in particular $\Siu$, lies close to the span of those points.  

To motivate this fact, consider the rigid case when $u:\R^m\to X$ is homogeneous about points spanning an $m-2$ dimensional subspace $V$ (i.e. the pinching is $0$ at these points).  In this case, $u$ cannot have any points of order greater than $1$ outside of $V$, for if $x$ were such a point, then for any point $y$ on any line from $x$ to $V$, $y$ would also have high order, and we would have that $\A$ had Hausdorff dimension at least $(m-1)$.  This, of course, contradicts Theorem \ref{thm:BigGS}.

First, let us make the notion of points being ``spread out enough" more precise.

\begin{definition}
We say that a set of points $\{x_i\}_{i=0}^k\subset B_r(x)$ is {\bf $\rho r$-linearly independent} if for each $i=1,2,\dots,k$,
\begin{equation}
d(x_i,x_0+\text{span}\{x_{i-1}-x_0,\dots,x_1-x_0\})\geq\rho r.
\end{equation}
\end{definition}
\begin{definition}
We say that a set $F\subset B_r(x)$ {\bf $\rho r$-spans} a $k$-dimensional affine subspace $V$ if there is a $\rho r$-linearly independent set $\{x_i\}_{i=0}^k\subset F$ so that $V=x_0+\text{span}\{x_i-x_0\}$.
\end{definition}

We also note, at this point, that if $F\cap B_r(x)$ fails to $\rho r$-span any $k$-dimensional affine subspaces, then it lies in $B_{\rho r}(L)$ for some $(k-1)$-dimensional affine subspace $L$.  Indeed, if we take $\kappa$ to be the maximum number so that $F$ has a $\rho r$-linearly independent set of size $\kappa$, we can take the span of any such set and call it $L$.  Then we must have that $\kappa<k$, but moreover $F\subset B_{\rho r}(L)$, for if $y\in F\setminus B_{\rho r}(L)$, by the definition of $\rho r$-span, the set $\{x_0,\dots,x_\kappa,y\}$ $\rho r$-spans a $(\kappa+1)$-dimensional affine subspace.

\begin{lemma}\label{lemma6.4}
Suppose that $u$ satisfies Assumptions \ref{cone} and \ref{bound}, and let $\rho_1,\rho_2,\rho_3$ be given.  Then there is an $\epsilon=\epsilon(m,X,\Lambda,\rho_1,\rho_2,\rho_3)>0$ such that the following statement holds.

If $\{x_i\}_{i=0}^{m-2}\subset B_1(0)$ is a $\rho_1$-linearly independent set with
\[
W_{\rho_2}^2(x_i)=I_\phi(x_i,2)-I_\phi(x_i,\rho_2)<\epsilon
\]
for each $i$, then
\[
\A\cap(B_1(0)\setminus B_{\rho_3}(V))=\varnothing,
\]
where $V$ is the affine span of $\{x_i\}_{i=0}^{m-2}$, $V=x_0+\text{span}\{x_i-x_0:1\leq i\leq m-2\}$.
\end{lemma}

The same hypotheses let us conclude that $I_\phi(x,r)$ is almost constant on $V$, as long as $r$ is not near $0$.  (In \cite{DMSV} more precise bounds are stated, which we will not need; such bounds should follow from computations similar to those in Theorem \ref{thm4.2}.)

\begin{lemma}\label{lemma6.5}
Suppose that $u$ satisfies Assumptions \ref{cone} and \ref{bound}, and let $\delta>0$ be given.  There exists $\epsilon=\epsilon(m,X,\Lambda,\rho_1,\rho_2,\rho_3,\delta)>0$ such that the following statement holds.

If $\{x_i\}_{i=0}^{m-2}\subset B_1(0)$ is a $\rho_1$-linearly independent set with
\[
W_{\rho_2}^2(x_i)=I_\phi(x_i,2)-I_\phi(x_i,\rho_2)<\epsilon
\]
then for $V=x_0+\text{span}\{x_i-x_0:1\leq i\leq m-2\}$, for all $y_1,y_2\in B_1(0)\cap V$, and for all $r_1,r_2\in[\rho_3,1]$,
\[
\big|I_\phi(y_1,r_1)-I_\phi(y_2,r_2)\big|\leq\delta.
\]
\end{lemma}

The proofs of both of these lemmas will proceed by compactness arguments, using Lemma \ref{compactness}.  We shall also need the following two lemmas, which provides more precise information in the rigid case, when the maps are conically homogeneous.

\begin{lemma}\label{orderconstancy}
If $u:\Omega\to X$ is a minimizing map into an $F$-connected complex, and $u$ is conically homogeneous about the point $x_0$, then $\text{Ord}_u(x)$ is constant along rays emanating from $x_0$.
\end{lemma}

\begin{lemma}\label{twohomogeneous}
Suppose that $\Omega\subset\R^m$ is convex open domain.  Suppose also that $u:\Omega\to X$ is a nonconstant continuous map into a conical $F$-connected complex which is conically homogeneous with respect to two distinct points $x_1$ and $x_2$, with orders $\alpha_1$ and $\alpha_2$, respectively.  That is,
\begin{align*}
u(x)&=|x-x_1|^{\alpha_1}u\bigg(\frac{x-x_1}{|x-x_1|}+x_1\bigg)&\text{for }x\neq x_1\\
u(x)&=|x-x_2|^{\alpha_2}u\bigg(\frac{x-x_2}{|x-x_2|}+x_2\bigg)&\text{for }x\neq x_2.
\end{align*}

Then $\alpha_1=\alpha_2$ and $u$ is independent of the $x_1-x_2$ direction, that is, \[
u(y+\lambda(x_1-x_2))=u(y)\] for any $y\in\R^m,\lambda\in\R$ for which both $y$ and $y+\lambda(x_1-x_2)$ lie in $\Omega$.  (Note that we could quite naturally extend such a conically homogeneous map to all of $\R^m$, however, removing this latter restriction.)  In particular, $u(\lambda x_1+(1-\lambda)x_2)=0_X$ for every $\lambda\in\R$.  (That is, $u$ is identically $0_X$ along the line determined by $x_1$ and $x_2$.)
\end{lemma}

\begin{proof}[Proof of Lemma \ref{orderconstancy}]
Without loss of generality, we suppose that $x_0=0$, so that two points $p,q$ are along the same ray emanating from $x_0$ if and only if $p=\lambda q$ for some $\lambda>0$.  We note that if $u$ is conically homogeneous of order $\alpha$, then $\nabla u$ is conically homogeneous of order $\alpha-1$.  Using this, the definition of the order, the definition of conical homogeneity, and changing variables to $y=\lambda x$, we see that for any $\lambda>0$,
\begin{align*}
\text{Ord}_u(z,r,u(z))&=\frac{r\int_{B_{r}(z)}\lambda^{2-2\alpha}|\nabla u(\lambda x)|^2\phantom{i}dx}{\int_{\partial B_r(z)}\lambda^{-2\alpha}d^2(u(\lambda x),u(\lambda z))\phantom{i}d\Sigma(x)}\\
&=\frac{r\lambda^{2-n-2\alpha}\int_{B_{\lambda r}(\lambda z)}|\nabla u(y)|^2\phantom{i}dy}{\lambda^{-n+1-2\alpha}\int_{\partial B_{\lambda r}(\lambda z)}d^2(u(y),u(\lambda z))\phantom{i}d\Sigma(y)}\\
&=\frac{\lambda r\int_{B_{\lambda r}(\lambda z)}|\nabla u(y)|^2\phantom{i}dy}{\int_{\partial B_{\lambda r}(\lambda z)}d^2(u(y),u(\lambda z))\phantom{i}d\Sigma(y)}\\
&=\text{Ord}_u(\lambda z,\lambda r,u(\lambda z)).
\end{align*}

Hence,
\[
\text{Ord}_u(z)=\lim_{r\to0}\text{Ord}_u(z,r,u(z))=\lim_{r\to0}\text{Ord}_u(\lambda z,\lambda r,u(\lambda z))=\text{Ord}_u(\lambda z)
\]
for any $\lambda>0$.  This demonstrates that the order is constant along rays emanating from $x_0=0$, as desired.
\end{proof}

\begin{proof}[Proof of Lemma \ref{twohomogeneous}]
We first remark that $u(x_1)=u(x_2)=0_X$, by conical homogeneity and continuity.  By homogeneity, we thus see that $u(x)\equiv0_X$ along the line passing through these points.

By homogeneity, it suffices to show that $u$ is constant along lines in the $x_2-x_1$ direction which lie close to the line joining $x_1$ to $x_2$.  For lines further away, we can simply rescale towards either $x_1$ or $x_2$, use the invariance along lines close to $x_1$ and $x_2$, and then rescale back up.  In particular, let $\epsilon$ be such that the cylinder of radius $2\epsilon$ about the line segment between $x_1$ and $x_2$ is contained in $\Omega$. 

Let $v$ be a vector perpendicular 
to $x_2-x_1$, of length less than $\epsilon$.  Consider the points $x_1+v$, $x_1+2v$, $x_2+v$, $x_2+2v$, $\frac{x_1+x_2}{2}+\frac{v}{2}=:m_1$, and $\frac{x_1+x_2}{2}+v=:m_2$.  Suppose that $u(x_1+v)=P$ and $u(x_2+v)=Q$.  (See Figure \ref{fig:homogeneity}.)

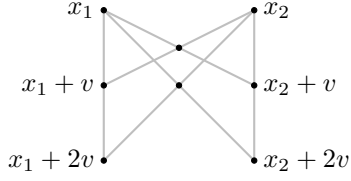
\begin{figure}
\centering
\begin{tikzpicture}
\draw[gray!50,thick] (-1,0)--(1,-1);
\draw[gray!50,thick] (-1,0)--(1,-2);
\draw[gray!50,thick] (1,0)--(-1,-1);
\draw[gray!50,thick] (1,0)--(-1,-2);
\draw[gray!50,thick] (-1,0)--(-1,-2);
\draw[gray!50,thick] (1,0)--(1,-2);
\filldraw[black] (-1,0) circle (1pt) node [anchor=east]{$x_1$};
\filldraw[black] (1,0) circle (1pt) node [anchor=west]{$x_2$};
\filldraw[black] (-1,-1) circle (1pt) node [anchor=east]{$x_1+v$};
\filldraw[black] (1,-1) circle (1pt) node [anchor=west]{$x_2+v$};
\filldraw[black] (-1,-2) circle (1pt) node [anchor=east]{$x_1+2v$};
\filldraw[black] (1,-2) circle (1pt) node [anchor=west]{$x_2+2v$};
\filldraw[black] (0,-0.5) circle (1pt) node [anchor=south] {};
\filldraw[black] (0,-1) circle (1pt) node [anchor=north] {};
\end{tikzpicture}
\caption{The points $x_1,x_1+v,x_1+2v$ are marked, in a vertical line, and the points $x_2,x_2+v,x_2+2v$ in another vertical line, so that $x_1,x_2$ are in the same horizontal level (similarly for $x_1+v, x_2+v$, and similarly for $x_1+2v,x_2+2v$).  Lines are drawn to connect $x_1$ to $x_2+v$ and $x_2$ to $x_1+v$; the shared midpoint of these lines is marked; this point is $\frac{x_1+x_2}{2}+\frac{v}{2}$.  Similarly, lines are drawn to connect $x_1$ to $x_2+2v$ and $x_2$ to $x_1+2v$; the shared midpoint of these lines is marked; this point is $\frac{x_1+x_2}{2}+v$.  In the first part of this argument, we rescale along the six lines shown to draw conclusions about the map $u$, which is assumed to be homogeneous about both $x_1$ and $x_2$.}
\label{fig:homogeneity}
\end{figure}

Using the conical homogeneity of $u$, we see at once that $u(x_1+2v)=2^{\alpha_1}P$ and $u(x_2+2v)=2^{\alpha_2}Q$.  Because $\frac{x_1+x_2}{2}+\frac{v}{2}$ is both the midpoint of the line joining $x_1$ and $x_2+v$ and of the line joining $x_2$ to $x_1+v$, we see that
\begin{equation}\label{doublemidpoint}
2^{-\alpha_1}Q=u\Big(\frac{x_1+x_2}{2}+\frac{v}{2}\Big)=2^{-\alpha_2}P.
\end{equation}

Finally, using the fact that $\frac{x_1+x_2}{2}+v$ is (similarly) the midpoint of both the line joining $x_1$ to $x_2+2v$, and of the line connecting $x_2$ and $x_1+2v$, we see that
\[
2^{-\alpha_2}2^{\alpha_1}P=u\Big(\frac{x_1+x_2}{2}+v\Big)=2^{-\alpha_1}2^{\alpha_2}Q=P
\]
where in the last equality we have used (\ref{doublemidpoint}) in recalling that $2^{-\alpha_1}Q=2^{-\alpha_2}P$.

Hence we see that $2^{\alpha_2-\alpha_1}P=P$ and hence that $\alpha_2=\alpha_1$ or $P=Q=0_X$.  Because $u$ is nonconstant, by choosing $v$ so that one of $P$ or $Q$ is not $0_X$, we see that $\alpha_2=\alpha_1$.  We conclude, in either case, that $P=Q$.

It remains only to show that $u$ is independent of the $x_2-x_1$ direction.  So, suppose again that $v\in\R^m$ is perpendicular to $x_2-x_1$ and has length less than $\epsilon$.  We shall show that for any $\lambda\in\R$ for which $y:=x_1+v+\lambda(x_2-x_1)\in\Omega$, we have that $u(y)=u(x_1+v)=u(x_2+v)$.  If $\lambda\geq0.5$, we consider the point \[z:=x_1+\frac{v}{\lambda}+(x_2-x_1)=x_2+\frac{v}{\lambda}.\]
(See Figure \ref{fig:invariance}.)

By our choice of $v$ and the fact that $\lambda\geq0.5$, $z\in\Omega$.  By construction, $z$ lies on the ray connecting $x_1$ and $y$, and on the ray connecting $x_2$ and $x_2+v$.  In particular,
\[
z=x_1+\frac{1}{\lambda}(y-x_1)=x_2+\frac{1}{\lambda}v.
\]
We apply the conical homogeneity of $u$ to see that
\[
u(z)=\lambda^{-\alpha_2}u(x_2+v)=\lambda^{-\alpha_1}u(y).
\]
In particular, we conclude that $u(y)=u(x_2+v)=u(x_1+v)$.

\begin{figure}\centering
\begin{tikzpicture}
\draw[gray!50,thick] (-1,0)--(1,0);
\draw[gray!50,thick] (-1,-1)--(1,-1);
\draw[gray!50,thick] (1,0)--(1,-1.5);
\draw[gray!50,thick] (-1,0)--(1,-1.5);
\filldraw[black] (-1,0) circle (1pt) node [anchor=east] {$x_1$};
\filldraw[black] (1,0) circle (1pt) node [anchor=west] {$x_2$};
\filldraw[black] (-1,-1) circle (1pt) node [anchor=east] {$x_1+v$};
\filldraw[black] (1,-1) circle (1pt) node [anchor=west] {$x_2+v$};
\filldraw[black] (0.333333,-1) circle (1pt) node [anchor=north] {$y$};
\filldraw[black] (1,-1.5) circle (1pt) node [anchor=west] {$z$};
\end{tikzpicture}
\caption{The points $x_1$, $x_2$, $x_1+v$, $x_2+v$, $y$, and $z$ are marked.  One horizontal line connects $x_1$ to $x_2$, and another horizontal line connects $x_1+v$, $y$, and $x_2+v$.  A vertical line connects $x_2$ to $x_2+v$, and a diagonal line connects $x_1$ to $y$; the point $z$ is drawn at their intersection.    In the second part of this argument, we rescale along the drawn vertical and diagonal lines.}\label{fig:invariance}
\end{figure}
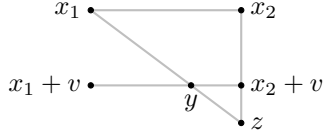

For $\lambda\leq0.5$, the same argument works with the roles of $x_1,x_2$ interchanged.  In particular taking
\[
z=x_1+\frac{v}{1-\lambda}=x_2+\frac{v}{1-\lambda}+(x_1-x_2)
\]
the rest of the argument follows analogously because in this case
\[
z=x_1+\frac{1}{1-\lambda}v=x_2+\frac{1}{1-\lambda}(y-x_2).
\]

In either case, $u(y)=u(x_1+v)=u(x_2+v)$, so that $u$ is indeed independent of the $x_2-x_1$ direction, as desired.

This result is similar to Lemma 6.8 of \cite{DMSV}, but our proof does not require that the homogeneous map to have domain $\R^m$, and is hence slightly more general.  The proof presented here is also of a more geometric flavor.
\end{proof}

\begin{proof}[Proof of Lemma \ref{lemma6.4}]
We assume for sake of contradiction that no such $\epsilon$ exists and conclude that there is a sequence of maps $u_q$ satisfying Assumptions \ref{cone} and \ref{bound} and a sequence of sets $P_q=\{x_{q,0},\dots,x_{q,m-2}\}\subset B_1(0)$ so that
\begin{itemize}
\item for each $q$, $P_q$ is $\rho_1$-linearly independent
\item $I_{\phi,u_q}(x_{q,i},2)-I_{\phi,u_q}(x_{q,i},\rho_2)\to0$ as $q\to\infty$
\item there is a point $y_q\in\A$ (i.e. $\text{Ord}_{u_q}(y_q)>1$) at least a distance of $\rho_3$ from the affine span of $P_q$.
\end{itemize}

We normalize $u$ so that $H_{\phi,u_q}(0,64)=1$.  Then, as $I_{\phi,u_q}(0,64)\leq\Lambda$, these maps have uniformly bounded energy.  Applying Lemma \ref{compactness}, we conclude that there is a subsequence (not relabeled) $u_q$ converging uniformly and strongly in $H^1(B_{16}(0))$ to a nonconstant minimizing map $u$ on the ball $B_{16}(0)$.

By extracting an appropriate subsequence, we may also assume that for each $0\leq i\leq m-2$, $x_{q,i}\to x_i$ and that $y_q\to y$.  Because each set $\{x_{q,0},\dots,x_{q,m-2}\}$ is $\rho_1$-linearly independent, the limit set $\{x_0,\dots,x_{m-2}\}$ is $\rho_1$-linearly independent as well, and similarly $y$ is at a distance of at least $\rho_3$ from their affine span $x_0+\text{span}\{x_1-x_0,\dots,x_{m-2}-x_0\}$.  Moreover, because each $y_q$ is a point of high order for the map $u_q$, the point $y$ is a point of high order for the map $u$ (by Lemma \ref{compactness}).

By the fact that the $u_q$ converge uniformly to $u$, and $x_{q,i}\to x_i$, we conclude that $H_{\phi,u_q}(x_{q,i},r)\to H_{\phi,u}(x_i,r)$ as $q\to\infty$, for each $0\leq i\leq m-2$ and for any $r<16-|x_i|$ (so that $B_r(x_i)$ is compactly contained in $B_{16}(0)$, that is).

Because the maps $u_q$ converge to $u$ strongly in $H^1(B_{16}(0))$, it is not hard to see that $E_{\phi,u_q}(x_{q,i},r)\to E_{\phi,u_q}(x_i,r)$ for each $0\leq i\leq m-2$ and for any $r<16-|x_i|$.

In particular, we see that for each $0\leq i\leq m-2$, $I_{\phi,u_q}(x_{q,i},r)\to I_{\phi,u}(x_i,r)$ as long as $r<16-|x_i|$; in particular this holds for any $r<15$ and hence for each of $r=2$ and $r=\rho_2$.  Recalling that $I_{\phi,u_q}(x_{q,i},2)-I_{\phi,u_q}(x_{q,i},\rho_2)\to0$, we conclude that for each $0\leq i\leq m-2$,
\[
I_{\phi,u}(x_i,2)=I_{\phi,u}(x_i,\rho_2).
\]
Applying Lemma \ref{homchar}, we conclude that $u$ is conically homogeneous about each $x_i$ on $B_2(x_i)$ and in particular on $B_1(0)$.  In particular, by Lemma \ref{orderconstancy}, the order of $u$ is constant along each ray emanating from any of the $(m-1)$ points $x_i$.

Thus, the points along the ray from $x_0$ to $y$ are all of order greater than $1$ (including $x_0$ itself, by the upper semicontinuity of the order).  But then when we take the rays from these points to $x_1$, we obtain a two-dimensional set on which all of the points have high order (in the triangle with vertices at $y,x_0,$ and $x_1$).  Iterating this process for each of the $(m-1)$ points $x_i$, we see that the set
\[
\bigg\{y+\sum_{i=0}^{m-2}t_i(x_i-y)\text{ so that }0\leq t_i\leq1\bigg\}
\]
consists entirely of high-order points.  However, this set has Hausdorff dimension $m-1$, contradicting Theorem \ref{thm:BigGS}.  We thus conclude the desired result.
\end{proof}

\begin{proof}[Proof of Lemma \ref{lemma6.5}]
As above, we proceed by contradiction: suppose no such $\epsilon$ exists, and take a sequence $u_q$ with corresponding $P_q=\{x_{q,0},\dots,x_{q,m-2}\}\subset B_1(0)$ so that
\begin{itemize}
\item for each $q$, $P_q$ is $\rho_1$-linearly independent
\item $I_{\phi,u_q}(x_{q,i},2)-I_{\phi,u_q}(x_{q,i},\rho_2)\to0$ as $q\to\infty$
\item there are points $y_{q,1},y_{q,2}$ in $V_q$, and $r_{q,1},r_{q,2}\in[\rho_3,1]$ so that $|I_\phi(y_{q,1},r_{q,1})-I_\phi(y_{q,2},r_{q,2})|>\delta$
\end{itemize}
where $V_q$, of course, denotes $\{x_{0,q}+\text{span}\{x_{i,q}-x_{0,q}\}\}$.

Proceeding exactly as in the previous proof, we can choose an appropriate subsequence of the $u_q$ so that the $x_{q,i}\to x_i$ for $i=0,\dots,m-2$, and so that $y_{q,j},r_{q,j}\to y_j,r_j$ (respectively) for $j=1,2$.  We can also choose this subsequence to converge uniformly on $B_{16}(0)$ to a nonconstant map $u$.  Just as in the preceding argument, we then have that if $z_q\to z$, and $r<16-|z|$, then $I_{\phi,u_q}(z_q,r)\to I_{\phi,u}(z,r)$.

In particular, for each $x_i$, we have that $I_{\phi,u}(x_i,2)=I_{\phi,u}(x_i,\rho_2)$ so that $u$ is conically homogeneous about each $x_i$ on $B_1(0)$, by Lemma \ref{homchar}.  Moreover, for $y_1,y_2$, we have that $y_1,y_2\in V:=\{x_0+\text{span}\{x_i-x_0\}\}$, and that
\begin{equation}\label{6.5lowerbound}
|I_{\phi,u}(y_1,r_1)-I_{\phi,u}(y_2,r_2)|\geq\delta.
\end{equation}

On the other hand, by repeated applications of Lemma \ref{twohomogeneous}, we have that $u$ is invariant with respect to the plane spanned by $\{x_i-x_0\}_{i=1}^{m-2}$.  In particular, $u$ is homogeneous about every point of $V$, and $I_{\phi,u}(x,r)$ is constant as a function of $r$ for every $x$ in that plane.  Putting these facts together, we conclude that there is some $\alpha$ so that $I_{\phi,u}(x,r)=\alpha$ for every $x\in V$ and for every $r>0$.

In particular, we see that
\[
I_{\phi,u}(y_1,r_1)=I_{\phi,u}(y_2,r_2)
\]
in contradiction of (\ref{6.5lowerbound}), as desired.
\end{proof}

\section{Minkowski-Type Estimate}\label{Minkowski7}

We now use the previous estimates and Theorem \ref{thm:NV} to prove the Minkowski upper bound of Theorem \ref{MinkowskiBoundThm}.

We begin by restating Theorem \ref{thm:NV}, for convenience.

\begin{theorem}
Fix $k\leq m\in\N$, let $\{B_{s_j}(x_j)\}_{j\in J}\subseteq B_2(0)\subset\R^m$ be a sequence of pairwise disjoint balls centered in $B_1(0)$, and let $\mu$ be the measure
\[
\mu=\sum_{j\in J}s_j^k\delta_{x_j}.
\]

Then, there exist constants $\delta_0=\delta_0(m)$ and $C_R=C_R(m)$ depending only on $m$ such that if for all $B_r(x)\subset B_2(0)$ with $x\in B_1(0)$ we have the integral bound
\[
\int_{B_r(x)}\bigg(\int_0^rD_{\mu}^k(y,s)\frac{ds}{s}\bigg)d\mu(y)<\delta_0^2r^k
\]
then the measure $\mu$ is bounded by
\[
\mu(B_1(0))=\sum_{j\in J}s_j^k\leq C_R.
\]
\end{theorem}

We shall use this theorem and our previous estimates to prove the following result.

\begin{proposition}\label{prop7.2}
Let $u$ satisfy Assumptions \ref{cone} and \ref{bound}.  Fix $x\in B_{1/8}(0)$ and $0<s<r\leq\frac{1}{8}$, and let $D\subset\A\cap B_r(x)$ be any subset of $\A\cap B_r(x)$.  Set $U=\text{sup}\{I_\phi(y,r):y\in D\}$.  Then there exist a positive $\delta=
\delta(m,X,\Lambda)$, a constant $C_V=C_V(m)\geq1$, a finite covering of $D$ with balls $B_{s_i}(x_i)$, and a corresponding decomposition of $D$ into sets $A_i\subset D$ so that:\begin{enumerate}
\item $A_i\subset B_{s_i}(x_i)$ and $s_i\geq s$,
\item $\sum_{i}s_i^{m-2}\leq C_Vr^{m-2}$,
\item for each $i$, either $s_i=s$, or
\begin{equation}\label{freqdrop}
\text{sup}\{I_\phi(y,s_i):y\in A_i\}\leq U-\delta.
\end{equation}
\end{enumerate}
\end{proposition}

With this, we can prove the desired Minkowski bound of Theorem \ref{MinkowskiBoundThm}.

\begin{proof}[Proof of Theorem \ref{MinkowskiBoundThm}]
Set $D_0=\Siu\cap B_1(0)$ and note that, by Lemma \ref{lemma3.4},
\[
U_0=\sup\{I_\phi(y,1/8):y\in D_0\}\leq C(\Lambda+1).
\]

We now begin to iteratively apply Proposition \ref{prop7.2}.  First, apply it with $r=1$, $s=\rho$, and $D=D_0$, obtaining $\{A_i\},\{B_{s_i}(x_i)\}$ for $i\in I_1$ as the corresponding decomposition of $D$.  Note that
\[
\sum_{i\in I_1}s_i^{m-2}\leq C_V.
\]

Let $I_1^g:=\{i:s_i=\rho\}$, these are the ``good indices" where the ball is already at the minimum desired radius of $\rho$.  If $s_i>\rho$, we have the frequency drop
\[
\sup\{I_\phi(y,s_i):y\in A_i\}\leq U_0-\delta.
\]

For such $i$, we apply the proposition again, with $r=s_i$, $s=\rho$, and $D=A_i$.  This gives us a decomposition $\{A_{i,j}\}$ of the $A_i$, with corresponding balls $B_{s_{i,j}}(x_{i,j})$, with the indices $j\in I_1^i$, and
\[
\sum_{j\in I_1^i}s_{i,j}^{m-2}\leq C_Vs_i^{m-2}.
\]

Taking all of the good balls from earlier, together with these new decompositions, we form $I_2$; more precisely $I_2$ is the union of $I_1^g$ with $I_1^i$ for each $i\notin I_1^g$.  This defines (after relabeling) a new decomposition of $\Siu\cap B_1(0)$ into sets $A_i$, with corresponding balls $B_{s_i}(x_i)$, where $i\in I_2$.  We have the bound
\[
\sum_{i\in I_2}s_i^{m-2}\leq C_V\sum_{i\in I_1}s_i^{m-2}\leq C_V^2.
\]

We also have, if $s_i>\rho$ for $i\in I_2$, then the frequency drop is
\[
\sup\{I_\phi(y,s_i):y\in A_i\}\leq U_0-2\delta.
\]

Iterating, for each $k$, we find a decomposition $\{A_i\}_{i\in I_k}$ with corresponding balls $\{B_{s_i}(x_i)\}_{i\in I_k}$, so that
\[
\sum_{i\in I_k}s_i^{m-2}\leq C_V^k
\]
and either $s_i=\rho$ or
\[
\sup\{I_\phi(y,s_i):y\in A_i\}\leq U_0-k\delta.
\]

But since the frequency is always positive, the latter case cannot occur after $\kappa=\floor{\delta^{-1}U_0}+1$ applications of the process; at this stage, we must have that $s_i=\rho$ for all $i\in I_\kappa$.  Therefore, $\{B_\rho(x_i)\}_{i\in I_\kappa}$ is a family of $N$ balls so that $N\rho^{m-2}\leq C_V^\kappa=C(m,X,\Lambda)$, which cover $\Siu\cap B_1(0)$.  We can also see that $B_\rho(\Siu\cap B_{1/8}(0))\subset\bigcup_{i}B_{2\rho}(x_i)$ and hence conclude that
\[
|B_\rho(\Siu\cap B_{1/8}(0))|\leq2^m\rho^mN\leq C\rho^2.
\]
\end{proof}

To prove Proposition \ref{prop7.2}, we use an intermediate covering.

\begin{lemma}\label{lemma7.3}
Let $u$ satisfy Assumptions \ref{cone} and \ref{bound}, $\rho\leq\frac{1}{128}$, and $0<\sigma<\tau\leq\frac{1}{8}$ be positive reals.  Let $D$ be any subset of $\A\cap B_\tau(0)$ and set $U:=\sup_{y\in D}I_\phi(y,\tau)$.  Then there exist a $\delta=\delta(m,X,\Lambda,\rho)$, a constant $C_R=C_R(m)$, and a covering of $D$ by balls $B_{r_i}(x_i)$ with the following properties:
\begin{enumerate}
\item $r_i\geq10\rho\sigma$,
\item $\sum_{i\in I}r_i^{m-2}\leq C_R\tau^{m-2}$
\item For each $i$, either $r_i\leq\sigma$, or the set of points
\begin{equation}
F_i=D\cap B_{r_i}(x_i)\cap\{y:I_\phi(y,\rho r_i)>U-\delta\}
\end{equation}
is contained in $B_{\rho r_i}(L_i)\cap B_{r_i}(x_i)$ for some $(m-3)$-dimensional affine subspace $L_i$.
\end{enumerate}
\end{lemma}
\begin{proof}
First, without loss of generality we may assume that $\tau=\frac{1}{8}$ and $x=0$.  Although this may increase $I_\phi(0,64)$, it will still be bounded in terms of $\Lambda$ by Lemma \ref{lemma3.4}.  In the following, we will treat $\delta$ as fixed and determine the conditions on $\delta$ that will be necessary.  These will be met so long as $\delta$ is sufficiently small.


We shall first construct a covering of $D$ by an iterative process, starting with $\mathcal{C}(0)=\{B_{1/8}(0)\}$ and, at each step, we shall modify our covering by keeping some ``bad" balls (which we do not further refine in this lemma) and by refining the covering on the other ``good" balls.  We shall do this in such a way that $\mathcal{C}(k)$, the covering at the $k^{\text{th}}$ stage, satisfies the following properties:
\begin{itemize}
\item that the radii of the balls in $\mathcal{C}(k)$ are all of the form $\frac{1}{8}(10\rho)^j$ with $0\leq j\leq k$,
\item that if $B_r(x),B_{r'}(x')\in\mathcal{C}(k)$, then $B_{r/5}(x)\cap B_{r'/5}(x')=\varnothing$,
\item that if a ball $B\in\mathcal{C}(k)$ has a radius larger than $\frac{1}{8}(10\rho)^k$, then it is kept in $\mathcal{C}(k+1)$.
\end{itemize}

We note that to satisfy point (1) of our lemma's desired covering, which requires that all of the radii are at least $10\rho\sigma$, we merely need to stop this iterative process once some radii are less than $\sigma$.  This occurs at step $\kappa=-\floor{\log_{10\rho}(8\sigma)}$, which is the smallest natural number so that $\frac{1}{8}(10\rho)^\kappa\leq\sigma$.

Our inductive procedure is as follows.  Let $B_r(x)\in\mathcal{C}(k)$.  If $r=8^{-1}(10\rho)^j$ for $j<k$ (that is, if $B_r(x)$ is ``large"), we add it to $\mathcal{C}(k+1)$.  If $r=8^{-1}(10\rho)^k$, we consider the set
\[
F(B_r(x)):=D\cap B_r(x)\cap\{y:I_\phi(y,\rho r)>U-\delta\}.
\]

If $F(B_r(x))$ fails to $\rho r$-span an $(m-2)$-dimensional affine subspace, we say that $B_r(x)$ is ``bad" and add it to $\mathcal{C}(k+1)$; we note that $F(B_r(x))$ is contained in $B_{\rho r}(L)$ for some $(m-3)$-dimensional affine subspace $L$.

On the other hand, if $F(B_r(x))$ does $\rho r$-span an $(m-2)$ dimensional affine subspace $V$, we say that $B_r(x)$ is ``good" and discard it.  We will replace it with a new collection of (smaller) balls with radius $10\rho r$, and add these to $\mathcal{C}(k+1)$ instead.  

In this latter case, we recall Lemma \ref{lemma6.4}, and consider the $\epsilon(m,X,\Lambda,\rho,\rho,\rho)$ of this Lemma (where all of the $\rho_i$ of the Lemma have been set to $\rho$).  If we choose $\delta$ to be less than this $\epsilon(m,X,\Lambda,\rho,\rho,\rho)$, then we have that $D\cap B_r(x)$ is contained in $B_{\rho r}(V)$.

We take the collection of all good balls $\mathcal{G}(k)$ and enumerate them as $\{B_i\}$, with corresponding affine subspaces $V_i$.  Then, if we set
\[
G(k):=D\cap\bigcup_{i}B_{\rho r}(V_i)
\]
we see that this is just the part of $D$ that lies in our good balls.  We can cover $G(k)$ with a collection of balls of radius $10\rho r=8^{-1}(10\rho)^{k+1}$ so that the corresponding concentric balls of radii $2\rho r$ are pairwise disjoint, and we can do this so that the balls are all centered at points of $\bigcup_i (B_i\cap V_i)$, which will be important in the next step.  Call this new collection of balls $\mathcal{F}(k+1)$, and note that all balls in this collection have radius equal to $8^{-1}(10\rho)^{k+1}$. 

Consider the collection $\mathcal{B}(k)\subset\mathcal{C}(k)$ of the balls that have been kept in $\mathcal{C}(k+1)$; these are the ``bad balls" and the balls of radius larger than $8^{-1}(10\rho)^k$ (which were bad balls at some previous stage of our construction).  Let $\mathcal{B}_{1/5}(k)$ denote the corresponding collection of concentric balls shrunk by a factor of $\frac{1}{5}$.  If a ball $B\in\mathcal{F}(k+1)$ does not intersect any element of $\mathcal{B}_{1/5}(k)$, we add it into the covering $\mathcal{C}(k+1)$, otherwise we exclude it.  We note that this ensures that, for any two balls in $\mathcal{C}(k+1)$, the corresponding concentric balls of one-fifth the radius are disjoint (as desired).  We must check that $\mathcal{C}(k+1)$, as we have constructed it, is a cover of $D$.  Certainly $\mathcal{B}(k)\cup\mathcal{F}(k+1)$ covers $D$, by construction.  So, let $x\in D$.  If $x$ is contained in an element of $\mathcal{B}(k)$, we are done.  If $x\in B$ for some $B\in\mathcal{F}(k+1)$, on the other hand, we are also done if $B\in\mathcal{C}(k+1)$.  The only way for $B$ to be excluded from $\mathcal{C}(k+1)$ is if there is some ball $B_{r'}(x')\in\mathcal{B}(k)$ so that $B_{r'/5}(x')$ intersects $B$.  But the radius of $B$ is at most $\frac{r'}{10}$, so that in this case $B\subset B_{r'}(x')$.

Note, also, that if any of these balls contain no points of $D$, we may freely discard them (because they do not contribute to the cover).

Next, we have a pinching estimate.  In particular, for any $\eta>0$, for $\delta$ sufficiently small, then either
\begin{equation}\label{7.3pinchest}
\mathcal{C}(\kappa)=\{B_{1/8}(0)\}\hspace{5mm}
\text{or}\hspace{5mm}
I_\phi(x,\rho s/5)\geq U-\eta\text{ for each }B_s(x)\in\mathcal{C}(\kappa).
\end{equation}

This is because either the refining process stops immediately (in which case the former alternative holds) or, for any $B_s(x)\in\mathcal{C}(\kappa)$, we have that $s=8^{-1}(10\rho)^{j+1}$ for some $j\in\N$.  By our construction, we must have added this ball during the $j^{\text{th}}$ refinement stage.  This, in turn, means there is some good ball $B=B_{8^{-1}(10\rho)^j}(y)$ and some $(m-2)$-dimensional affine subspace $V$ so that $F(B)$ $\rho8^{-1}(10\rho)^j$-spans $V$, where $x\in V\cap B$.  Of course, there is some $z\in F(B)$ so that $z\in V\cap B$ as well.  Hence, from Lemma \ref{lemma6.5}, we can choose $\delta$ sufficiently small, depending on $\rho$ and $\eta$, to guarantee that
\[
|I_\phi(x,\rho s/5)-I_\phi(z,s)|\leq\frac{\eta}{2}
\]
and since $I_\phi(z,s)\geq U-\delta$, if we additionally require that $\delta<\frac{\eta}{2}$ we ensure that $I_\phi(x,\rho s/5)\geq U-\eta$.

Now, $\mathcal{C}(\kappa)$ is the desired covering, and to complete the proof it suffices to prove that
\[
\sum_{B_s(x)\in\mathcal{C}(\kappa)}s^{m-2}\leq C_R(m).
\]
We enumerate these balls in $\mathcal{C}(k)$ as $B_{5s_i}(x_i)$ for $i\in I$; note that when we enumerate them in this way, the corresponding $B_{s_i}(x_i)$ are pairwise disjoint.  With a mind to apply Theorem \ref{thm:NV}, we define the measures
\[
\mu=\sum_{i\in I}s_i^{m-2}\delta_{x_i}\hspace{5mm}\text{and}\hspace{5mm}\mu_s=\sum_{i\in I,s_i\leq s}s_i^{m-2}\delta_{x_i}.
\]

Clearly, we have that
\begin{itemize}
\item$\mu_t\leq\mu_\tau$ whenever $t\leq\tau$,
\item$\mu=\mu_{1/40}$,
\item setting $\overline{r}=\frac{1}{40}(10\rho)^\kappa$, we have $\mu_s=0$ for $s<\overline{r}$.
\end{itemize}

We seek to prove that $\mu_s(B_s(x))\leq C_R(m)s^{m-2}$ for every $s>0$ and for every $x\in B_{1/8}(0)$.  Now, if we set $\alpha=\floor{\log_2(\overline{r}^{-1}/8)}-6$, it will be enough to show
\begin{equation}\label{dymeas}
\mu_s(B_s(x))\leq C_R(m)s^{m-2}\text{ for all }x\text{ and for }s=\overline{r}2^j\text{ where }j=0,1,\dots,\alpha.
\end{equation}

First, note that unless $\{B_{5s_i}(x_i)\}$ is the initial cover $\{B_{1/8}(0)\}$, then all of the radii of balls in the cover are at most $\frac{10\rho}{40}\leq\frac{1}{512}$.  In particular, (\ref{dymeas}) would show that $\mu(B_{1/512}(x))\leq C_R(m)$ for any $x\in B_{1/8}(0)$, allowing us to conclude the desired packing estimate by covering $B_{1/8}(0)$ with finitely many balls of radius $\frac{1}{512}$.

We prove (\ref{dymeas}) by inducting on $j$.  The base case is relatively quick.  Since $\mu_{\overline{r}}(B_{\overline{r}}(x))=N(x,\overline{r})\overline{r}^{m-2}$, where $N(x,\overline{r})$ counts the number of balls $B_{s_i}(x_i)$ where $s_i=\overline{r}$ and $x_i\in B_{\overline{r}}(x)$.  We know that these balls are pairwise disjoint, and they all must be contained in $B_{2\overline{r}}(x)$, so that $N(x,\overline{r})\leq2^m$.

Now, we wish to show that if (\ref{dymeas}) holds for some $j<\alpha$, then it also holds for $j+1$.  We thus let $r=2^j\overline{r}$ and assume that for every $x$, $\mu_r(B_r(x))\leq C_R(m)r^{m-2}$, and seek to prove that $\mu_{2r}(B_{2r}(x))\leq C_R(m)(2r)^{m-2}$ for every $x$.

We shall first show the weaker bound
\begin{equation}\label{coarsebound}
\mu_{2r}(B_{2r}(x))\leq C(m)C_R(m)(2r)^{m-2}
\end{equation}
where $C(m)$ is some dimensional constant.  To do this, we note that
\[
\mu_{2r}=\mu_r+\sum_{i\in I,r<s_i\leq2r}s_i^{m-2}\delta_{x_i}=:\mu_r+\tilde{\mu}_r.
\]
Covering $B_{2r}(x)$ by $C(m)$ balls $B_r(x_j)$, the inductive assumption ensures that
\[
\mu_r(B_{2r}(x))\leq C(m)C_R(m)r^{m-2}.
\]
We also see that $\tilde{\mu}_r(B_{2r}(x))\leq N(x,2r)(2r)^{m-2}$, where $N(x,2r)$ counts the balls $B_{s_i}(x_i)$ with $r<s_i\leq 2r$.  The balls $B_r(x_i)$ must (still) be pairwise disjoint, and are all contained in $B_{3r}(x)$, so that $N(x,2r)\leq C(m)$ for some dimensional constant $C(m)$ as well.

We now want to upgrade the bound (\ref{coarsebound}) to
\begin{equation}\label{finebound}
\mu_{2r}(B_{2r}(x))\leq C_R(m)(2r)^{m-2}.
\end{equation}
For convenience, let $\mubar=\mu_{2r}\llcorner B_{2r}(x)$.  We shall seek to apply (a rescaled version of) Theorem \ref{thm:NV}; in particular if we can show that
\begin{equation}\label{7.3.5goal}
\int_{B_t(y)}\bigg(\int_0^tD_{\mubar}^{m-2}(z,s)\frac{ds}{s}\bigg)\phantom{i}d\mubar(z)<\delta_0^2t^{m-2}
\end{equation}
for all $y\in B_{2r}(x)$ and each $0<t\leq2r$, where $\delta_0$ is the constant that comes from Theorem \ref{thm:NV}, then we conclude that $\mubar(B_{2r}(x))\leq C_R(2r)^{m-2}$, where $C_R$ is the other constant from Theorem \ref{thm:NV}.  This is the desired bound.

To arrive at this estimate, we note that by (\ref{7.3pinchest}), we can assume that
\[
I_\phi(x,\rho s_i)\geq U-\eta
\]
for otherwise, the covering is simply $\mathcal{C}(\kappa)=\{B_{1/8}(0)\}$ in which case (\ref{dymeas}) is trivial.  That is, we can freely assume that the pinching of $x_i$ between the scales $\tau$ and $\rho s_i$ is controlled by $\eta$.

To exploit our pinching bounds, we define, for each $x_i\in\text{supp}(\mu)$
\[
\Wbar_s(x_i):=\begin{cases}
W_s^{32s}(x_i)&\text{if }s>s_i\\
0&\text{otherwise,}
\end{cases}
\]
and note that for each $i$ and for $0<s<1$,
\begin{equation}\label{mubarflat}
D_{\mubar}^{m-2}(x_i,s)\leq C(m,n,Q,\Lambda)s^{-(m-2)}\int_{B_s(x_i)}\Wbar_s(y)d\mubar(y).
\end{equation}

If $s<s_i$, this is simply $0\leq0$ because $\text{supp}(\mu)\cap B_s(x_i)=\{x_i\}$ in this case.  Otherwise, this is a consequence of Proposition \ref{flatbound}.

Let $t\leq2r$.  The estimate (\ref{mubarflat}) lets us bound
\begin{align*}
I:=\int_{B_t(y)}\bigg(\int_0^tD_{\mubar}^{m-2}(z,s)\frac{ds}{s}\bigg)\phantom{i}d\mubar(z)\leq C\int_{B_t(y)}\int_0^ts^{1-m}\int_{B_s(z)}\Wbar_s(\zeta)\phantom{|}d\mubar(\zeta)\phantom{|}ds\phantom{|}d\mubar(z)\\
=C\int_0^ts^{1-m}\int_{B_t(y)}\int_{B_s(z)}\Wbar_s(\zeta)\phantom{|}d\mubar(\zeta)\phantom{|}d\mubar(z)\phantom{|}ds.
\end{align*}

In this estimate we may intersect the domains of integration with $B_{2r}(x)$, as $\text{supp}(\mubar)\subset B_{2r}(x)$.  Moreover, we may integrate with respect to $\mu_s$ instead of $\mubar$ in both integrals.  In the $\zeta$ integral, if $\zeta\in\text{supp}(\mubar)\setminus\text{supp}(\mu_s)$ then $\zeta=x_i$ for some $i\in I$ for which $s_i>s$; hence $\Wbar_s(x_i)=0$ by definition.  On the other hand, if $z\in\text{supp}(\mubar)\setminus\text{supp}(\mu_s)$, then similarly $z=x_i$ with $s_i>s$, which means that $B_s(z)$ does not contain any other $x_i$, and we again see that $\Wbar_s(z)=0$.  Hence the $\zeta$ integral vanishes for such $z$.  So, we instead integrate with respect to $\mu_s$ and apply Fubini-Tonelli:
\[
I\leq C\int_0^ts^{1-m}\int_{B_{t+s}(y)\cap B_{2r}(x)}\Wbar_s(\zeta)\int_{B_s(\zeta)\cap B_{2r}(x)}d\mu_s(z)\phantom{|}d\mu_s(\zeta)\phantom{|}ds.
\]

For $s\leq r$ we recall the inductive assumption (\ref{dymeas}), and for $r<s\leq2r$ we have the coarse bound (\ref{coarsebound}).  Combining these, the innermost integral is bounded by $C(m)s^{m-2}$, so that
\begin{align}
I&\leq C(m,X,\Lambda)\int_0^t\int_{B_{t+s}(y)\cap B_{2r}(x)}\Wbar_s(\zeta)\phantom{|}d\mu_s(\zeta)\frac{ds}{s}\notag\\
&\leq C\int_0^t\int_{B_{t+s}(y)\cap B_{2r}(x)}\Wbar_s(\zeta)\phantom{|}d\mu_t(\zeta)\frac{ds}{s}\notag\\
&\leq C\int_{B_{2t}(y)}\int_0^t\Wbar_s(\zeta)\frac{ds}{s}d\mu_t(\zeta).\label{7.3penult}
\end{align}

Fix some $\zeta\in\text{supp}(\mu_t)$, so that $\zeta=x_i$ for some $i$.  Recall that by definition $\Wbar_s(\zeta)=0$ when $s<s_i$ and otherwise $\Wbar_s(x_i)=I_\phi(x_i,32s)-I_\phi(x_i,s)$.  Let $N$ be the largest integer so that $2^Ns_i\leq t$, and note that because $t\leq2r\leq\frac{1}{512}$, we then have that $32\cdot2^{N+1}s\leq\frac{1}{8}$.  We then bound this inner integral,
\begin{align}
\int_0^t\Wbar_s(\zeta)\frac{ds}{s}&=\int_{s_i}^t\Wbar_s(x_i)\frac{ds}{s}=\int_{s_i}^t(I_\phi(x_i,32s)-I_\phi(x_i,s))\frac{ds}{s}\notag\\
&\leq\sum_{j=0}^N\int_{2^js_i}^{2^{j+1}s_i}(I_\phi(x_i,32s)-I_\phi(x_i,s))\frac{ds}{s}\notag\\
&\leq\sum_{j=0}^N(I_\phi(x_i,32\cdot2^{j+1}s_i)-I_\phi(x_i,2^js_i))\int_{2^js_i}^{2^{j+1}s_i}\frac{ds}{s}\notag\\
&=\log2\sum_{j=0}^N(I_\phi(x_i,2^{j+6}s_i)-I_\phi(x_i,2^js_i))\notag\\
&=\log2\sum_{\ell=0}^5\sum_{j=0}^N(I_\phi(x_i,2^{j+\ell+1}s_i)-I_\phi(x_i,2^{j+\ell}s_i))\notag\\
&=\log2\sum_{\ell=0}^5(I_\phi(x_i,2^{N+\ell+1}s_i)-I_\phi(x_i,2^\ell s_i))\notag\\
&\leq6\log2(I_\phi(x_i,\tfrac{1}{8})-I_\phi(x_i,s_i))\stackrel{\text{(\ref{7.3pinchest})}}{\leq}(6\log2)\eta\label{etabound}
\end{align}

Now, we note that $\mu_t(B_{t}(x))\leq C(m)t^{m-2}$ for every $x$, where we (as above) invoke (\ref{dymeas}) for $t\leq r$ and (\ref{coarsebound}) for $t\leq2r$.  Covering $B_{2t}(y)$ by finitely many such balls, we see that $\mu_t(B_{2t}(y))\leq C(m)t^{m-2}$ as well.  With (\ref{7.3penult}) and (\ref{etabound}), we conclude that
\begin{equation}\label{7.3final}
\int_{B_t(y)}\bigg(\int_0^tD_{\mubar}^{m-2}(z,s)\frac{ds}{s}\bigg)d\mubar(z)\leq C(m,X,\Lambda)\eta t^{m-2}.
\end{equation}

If we take $\delta$ sufficiently small, we can make $\eta$ sufficiently small that (\ref{7.3.5goal}) holds.  This completes the proof of (\ref{dymeas}) and of the lemma.
\end{proof}

We conclude this section with the proof of Proposition \ref{prop7.2}.
\begin{proof}[Proof of Proposition \ref{prop7.2}]
As in the previous lemma, we note that we can take $x=0$ and $r=\frac{1}{8}$ without loss of generality.  We again use an inductive process to construct the desired covering, where Lemma \ref{lemma7.3} is used to form the intermediate coverings.  The $\rho$ of Lemma \ref{lemma7.3} will be chosen at the end of this inductive process.

First, apply Lemma \ref{lemma7.3} with $\tau=\frac{1}{8}$ and $\sigma=s$, and obtain the corresponding covering $\{B_{r_i}(x_i)\}=:\mathcal{C}(0)$.  By the dichotomy of Lemma \ref{lemma7.3}, each such ball either has $r_i\leq s$, or the points of high order in that ball lie close to an affine $(m-3)$-plane.  Thus, we define subcollections of $\mathcal{C}(0)$, $\mathcal{G}(0)=\{B_{r_i}(x_i):r_i\leq s\}$ and $\mathcal{B}(0)=\{B_{r_i}(x_i):r_i>s\}$.  For each $B_{r_i}(x_i)\in\mathcal{B}(0)$, we have the set $F_i$ and the affine subspace $L_i$ given by Lemma \ref{lemma7.3}.  For each $i$, we can cover $B_{2\rho r_i}(L_i)\cap B_{r_i}(x_i)$ by a number $N\leq C(m)\rho^{3-m}$ of balls of radius $4\rho r_i$.  If $4\rho r_i\leq s$, we are done refining these balls, we include these in $\mathcal{G}(1)$.  Otherwise, we include them in $\mathcal{B}(1)$ (and we will repeat the refining process).  We define $\mathcal{C}(1)=\mathcal{G}(1)\cup\mathcal{B}(1)$.  Observe that
\[
\sum_{B_{r_i}(x_i)\in\mathcal{C}(1)}r_i^{m-2}\leq C(m)\rho^{3-m}\sum_{B_{r_j}(x_j)\in\mathcal{C}(0)}(\rho r_j)^{m-2}=C(m)\rho\sum_{B_{r_j}(x_j)\in\mathcal{C}(0)}r_j^{m-2}.
\]
Thus, if we choose $\rho$ to be smaller than $\rho_0(m):=\frac{1}{2C(m)}$, which is a purely dimensional constant, we can guarantee that
\[
C(m)\rho\leq\frac{1}{2}.
\]
We now fix $\rho$ to be smaller than this dimensional constant and $\frac{1}{128}$ (to satisfy the hypotheses of Lemma \ref{lemma7.3}).

We proceed inductively.  With the collection $\mathcal{B}(n)=\{B_{r_i}(x_i)\in\mathcal{C}(n):r_i>s\}$ given, we apply Lemma \ref{lemma7.3} to each of these bad balls.  For each such ball, we have that the $F_i$ from the lemma lies in $B_{\rho r_i}(L_i)\cap B_{r_i}(x_i)$ for the affine subspace $L_i$ from the lemma, and we again cover $B_{2\rho r_i}(L_i)\cap B_{r_i}(x_i)$ with a finite number $N\leq C(m)\rho^{3-m}$ of balls with radius $4\rho r_i$, where $C(m)$ is the same dimensional constant as in the previous paragraph.  Once again, if $4\rho r_i<s$, we include these new balls in $\mathcal{G}(n+1)$, and we otherwise include them in $\mathcal{B}(n+1)$, with $\mathcal{C}(n+1)$ being the union of these collections.

After finitely many stages, we will reach a $\mathcal{C}(k)$ with no balls of radius greater than $s$.  We now take $\mathcal{C}=\bigcup_{j=0}^k\mathcal{C}(k)$ and note that
\[
\sum_{B_{r_i}(x_i)\in\mathcal{C}}r_i^{m-2}\leq\sum_{\ell=0}^k2^{-\ell}\sum_{B_{r_j}(x_j)\in\mathcal{C}(0)}r_j^{m-2}\leq2C_R(m).
\]

We now define the sets $A_i'$ for $B_{r_i}(x_i)$.  Let $0\leq n\leq k$, and define
\begin{itemize}
\item for $B_{r_i}(x_i)\in\mathcal{G}(n)$, we have $r_i\leq s$ and we simply set $A_i'=D\cap B_{r_i}(x_i)$;
\item for $B_{r_i}(x_i)\in\mathcal{B}(n)$, we set $A_i'=(D\cap B_{r_i}(x_i))\setminus F_i$, where $F_i$ is the set coming from Lemma \ref{lemma7.3}.
\end{itemize}

We note that for e.g. $n=0$, the $F_i$ are covered by $\mathcal{C}(1)$, so that
\[
D\subset\bigcup_{B_{r_i}(x_i)\in\mathcal{C}(0)}A_i'\bigcup_{B_{r_i}(x_i)\in\mathcal{C}(1)}B_{r_i}(x_i).
\]
Of course, we then note that the $F_i$ for the bad balls of $\mathcal{C}(1)$ are, in turn, covered by $\mathcal{C}(2)$.  This lets us replace the latter union by $\bigcup_{B_{r_i}(x_i)\in\mathcal{C}(1)}A_i'\bigcup_{B_{r_i}(x_i)\in\mathcal{C}(2)}B_{r_i}(x_i)$.  Proceeding inductively, we conclude that
\[
D\subset\bigcup_{B_{r_i}(x_i)\in\mathcal{C}}A_i'
\]
because for $\mathcal{C}(k)$, all of the balls have radius at most $s$, so that for such balls $A_i'=B_{r_i}(x_i)\cap D$.  Hence the collection $A_i'$ does indeed form a covering of $D$.

Finally, by definition, either $r_i\leq s$ or
\[
\sup\{I_\phi(y,\rho r_i):y\in A_i'\}\leq U-\delta
\]
which is almost the same as the desired (\ref{freqdrop}), having only an additional factor of $\rho$ in the second argument of $I_\phi$.  Since $A_i'\subset B_{r_i}(x_i)$, we simply cover this (larger) set by $C(m)\rho^{-m}=C(m)$ balls $B_{\rho r_i}(x_{ij})$, where we recall that $\rho$ itself depends only on $m$.  Setting $A_{ij}=B_{\rho r_i}(x_{ij})\cap A_i'$, we obtain (\ref{freqdrop}) for these new balls, and retain the packing estimate up to multiplying by $C(m)$.

Lastly, some of the balls have radius less than $s$, but all have radii at least $10\rho s$, so replacing these with balls of radius $s$ completes the proposition, again multiplying the packing bound by another dimensional constant $C(m)$.
\end{proof}

\section{Rectifiability}\label{Rectifiability8}

In this section, we prove Theorem \ref{MicroRectThm}.  To do this, we recall Theorem \ref{thm:AT}.

\begin{theorem}[\cite{AT}, Corollary 1.3]
Let $S\subset\R^n$ be $\mathcal{H}^k$-measurable with $\mathcal{H}^k(S)<\infty$ and consider $\mu=\mathcal{H}^k\llcorner S$.  Then $S$ is countably $k$-rectifiable if and only if
\[
\int_0^1D_\mu^k(x,s)\frac{ds}{s}<\infty\hspace{5mm}\text{for }\mu\text{-a.e. }x.
\]
\end{theorem}


We prove the rectifiability of $\Siu\cap B_{1/8}$ using an argument similar to the proof of Lemma \ref{lemma7.3}.

\begin{proof}[Proof of Theorem \ref{MicroRectThm}]
From Theorem \ref{MinkowskiBoundThm}, we know that $\mu:=\mathcal{H}^{m-2}\llcorner(\Siu\cap B_{1/8}(0))$ is a finite Radon measure.  Moreover, by rescaling, we obtain the estimate
\begin{equation}\label{hausdorffbound}
\mu(B_r(x))\leq C(m,X,\Lambda)r^{m-2}.
\end{equation}
Arguing as in Lemma \ref{lemma7.3} we use Proposition \ref{flatbound} to derive the bound
\begin{align}
\int_{B_t(y)}\int_0^tD_\mu^{m-2}(z,s)\frac{ds}{s}\phantom{|}d\mu(z)&\leq\int_{B_t(y)}\int_0^ts^{1-m}\int_{B_s(z)}W_s^{32s}(\zeta)d\mu(\zeta)\phantom{|}ds\phantom{|}d\mu(z)\notag\\
&=C\int_0^ts^{1-m}\int_{B_t(y)}\int_{B_s(z)}W_s^{32s}(\zeta)d\mu(\zeta)\phantom{|}d\mu(z)\phantom{|}ds\notag\\
&\leq C\int_0^ts^{1-m}\int_{B_{t+s}(y)}W_s^{32s}(\zeta)\int_{B_s(\zeta)}d\mu(z)\phantom{|}d\mu(\zeta)\phantom{|}ds\notag\\
&\stackrel{\text{(\ref{hausdorffbound})}}{\leq}C\int_0^t\frac{1}{s}\int_{B_{t+s}(y)}W_s^{32s}(\zeta)d\mu(\zeta)\phantom{|}ds\notag\\
&\leq\int_{B_{2t}(y)}\int_0^tW_s^{32s}\frac{ds}{s}d\mu(\zeta)\label{thm2.6intbound}.
\end{align}

Next, as in the proof of (\ref{etabound}), we observe that
\[
\int_0^tW_s^{32s}(\zeta)\frac{ds}{s}\leq6\log2(I_\phi(\zeta,\tfrac{1}{8})-I_\phi(\zeta,0))\leq C(m,X,\Lambda)
\]
as long as $32t<\frac{1}{8}$.  Using this estimate along with (\ref{hausdorffbound}) and (\ref{thm2.6intbound}), we see that
\[
\int_{B_t(y)}\int_0^tD_\mu^{m-2}(z,s)\frac{ds}{s}d\mu(z)<\infty
\]
for any $y\in B_{1/8}(0)$ and any $t<\frac{1}{256}$.  This, along with crude bounds on $D_\mu^{m-2}(z,r)$ for $r>\frac{1}{512}$, allows us to apply Theorem \ref{thm:AT} to conclude that $\Siu\cap B_{1/8}(0)$ is indeed rectifiable.
\end{proof}

\bibliographystyle{amsplain}
\bibliography{BuildingsRectifiability}

\providecommand{\bysame}{\leavevmode\hbox to3em{\hrulefill}\thinspace}
\providecommand{\MR}{\relax\ifhmode\unskip\space\fi MR }
\providecommand{\MRhref}[2]{%
  \href{http://www.ams.org/mathscinet-getitem?mr=#1}{#2}
}
\providecommand{\href}[2]{#2}
\begin{thebibliography}{10}

\bibitem{A}
Onur Alper, \emph{On the singular set of free interface in an optimal partition
  problem}, Communications on Pure and Applied Mathematics \textbf{73} (2020),
  no.~4, 855--915.

\bibitem{AT}
Jonas Azzam and Xavier Tolsa, \emph{Characterization of n-rectifiability in
  terms of {J}ones' square function: Part ii}, Geometric and Functional
  Analysis \textbf{25} (2015), no.~5, 1371--1412.

\bibitem{CN}
Jeff Cheeger and Aaron Naber, \emph{Lower bounds on {R}icci curvature and
  quantitative behavior of singular sets}, Inventiones mathematicae
  \textbf{191} (2013), no.~2, 321--339.

\bibitem{DMDMComplex}
Georgios Daskalopoulos and Chikako Mese, \emph{On the singular set of harmonic
  maps into {DM}-complexes}, Memoirs of the American Mathematical Society
  \textbf{239} (2016).

\bibitem{DMEssReg}
\bysame, \emph{Essential regularity of the model space for the
  {W}eil--{P}etersson metric}, Journal f{\"u}r die reine und angewandte
  Mathematik (Crelles Journal) \textbf{2019} (2019), no.~750, 53--96.

\bibitem{DMRigidTeich}
\bysame, \emph{Rigidity of {T}eichm{\"u}ller space}, Inventiones mathematicae
  \textbf{224} (2021), no.~3, 791--916.

\bibitem{DMVSuperrigid}
Georgios Daskalopoulos, Chikako Mese, and Alina Vdovina, \emph{Superrigidity of
  hyperbolic buildings}, Geometric and Functional Analysis \textbf{21} (2011),
  no.~4, 905.

\bibitem{DMSV}
Camillo De~Lellis, Andrea Marchese, Emanuele Spadaro, and Daniele Valtorta,
  \emph{Rectifiability and upper {M}inkowski bounds for singularities of
  harmonic q-valued maps}, Commentarii Mathematici Helvetici \textbf{93}
  (2016).

\bibitem{ENV16}
Nick Edelen, Aaron Naber, and Daniele Valtorta, \emph{Quantitative {R}eifenberg
  theorem for measures}, arXiv: Classical Analysis and ODEs (2016).

\bibitem{ENV19}
\bysame, \emph{Effective reifenberg theorems in hilbert and banach spaces},
  Mathematische Annalen \textbf{374} (2019), no.~3, 1139--1218.

\bibitem{Fe}
Herbert Federer, \emph{Geometric measure theory: Reprint of the 1969 edition},
  Springer-Verlag, 1996.

\bibitem{GS}
Mikhail Gromov and Richard Schoen, \emph{Harmonic maps into singular spaces and
  p-adic superrigidity for lattices in groups of rank one}, Publications
  Math{\'e}matiques de l'Institut des Hautes {\'E}tudes Scientifiques
  \textbf{76} (1992), no.~1, 165--246.

\bibitem{KSSobolev}
Nicholas~J. Korevaar and Richard Schoen, \emph{Sobolev spaces and harmonic maps
  for metric space targets}, Communications in Analysis and Geometry \textbf{1}
  (1993), 561--659.

\bibitem{DS}
Camillo~De Lellis and Emanuele Spadaro, \emph{Regularity of area minimizing
  currents {III}: blow-up}, Annals of Mathematics \textbf{183} (2016), no.~2,
  577--617.

\bibitem{NV}
Aaron Naber and Daniele Valtorta, \emph{Rectiflable-{R}eifenberg and the
  regularity of stationary and minimizing harmonic maps}, Annals of Mathematics
  \textbf{185} (2017), no.~1, 131--227.

\end{thebibliography}
\end{document}